\documentclass[final,12pt,3p]{elsarticle}

\usepackage{graphicx}
\usepackage{amssymb}
\usepackage{amsthm}
\usepackage{amsmath}
\usepackage{longtable} 
\usepackage{tikz}
\usepackage{lineno}
\usepackage{verbatim} 
\usepackage{hyperref}
\usepackage{xcolor}
\usepackage{caption}
\usetikzlibrary{shapes}
\usepackage{subcaption}
\usepackage{mathrsfs} 
\usepackage{array}
\usepackage{multirow}
\usepackage{booktabs} 
\usepackage{adjustbox}
\usepackage[title]{appendix}

\hypersetup{
colorlinks=true,
linkcolor=red,
filecolor=green,
urlcolor=green,
anchorcolor=green,
citecolor=cyan,
pdftitle={Overleaf Example},
pdfpagemode=FullScreen,
}
\usepackage[T1]{fontenc}
\usepackage[utf8]{inputenc} 
\usepackage[english]{babel}
\usepackage{lmodern} 
\usepackage{csquotes} 

\usepackage{amsmath,amssymb,amsthm,mathtools}
\usepackage{mathrsfs,stmaryrd}
\usepackage{siunitx} 
\usepackage{graphicx} 
\usepackage{tikz}

\usepackage{circuitikz} %
\usetikzlibrary{angles,calc,quotes}
\usetikzlibrary {arrows.meta,bending,positioning} 

\usepackage{tikz-cd} 

\usepackage{float}
\usepackage{pgfplots} 
\pgfplotsset{compat=1.18}

\usepackage{enumitem} 
\usepackage{tabularx} 

\usepackage{pdfpages}
\usepackage[ruled]{algorithm2e}
\usepackage{listings}
\usepackage{bm}
\definecolor{keywordcolor}{RGB}{0,0,120}
\definecolor{commentcolor}{RGB}{0,100,0}
\definecolor{stringcolor}{RGB}{120,0,0}
\theoremstyle{plain}
\newtheorem{theorem}{Theorem}

\newtheorem{lemma}[theorem]{Lemma}

\newtheorem{problem}[theorem]{Problem}

\theoremstyle{definition}
\newtheorem{definition}{Definition}
\newtheorem{example}{Example}

\theoremstyle{remark}
\newtheorem{remark}{Remark}

\newcommand{\R}{\mathbb{R}}
\newcommand{\yd}[0]{y_{\rm data}}

\newcommand{\J}[0]{\mathcal{J}}
\newcommand{\mH}{\mathcal{H}}

\newcommand{\Uad}{\mathcal{U}_{\mathrm{ad}}}
\newcommand{\dd}{\mathrm{d}}

\DeclareMathOperator{\tr}{tr}
\DeclareMathOperator{\diag}{diag}

\usepackage[figcolor=white]{todonotes}

\newcommand{\Tmatthias}[1]{\todo[inline, color=blue!40]{Matthias: #1}}

\begin{document}

\begin{frontmatter}

\title{Structure-Preserving Data-Driven Identification of Port-Hamiltonian Differential-Algebraic Systems}

\author[BUW]{Tom Zwerschke}
\ead{tom.zwerschke@uni-wuppertal.de}

\author[BUW]{Matthias Ehrhardt\corref{Corr}}
\cortext[Corr]{Corresponding author}
\ead{ehrhardt@uni-wuppertal.de}

\author[BUW]{Michael Günther}
\ead{guenther@uni-wuppertal.de}

\author[BUWO]{Claudia Totzeck}
\ead{totzeck@uni-wuppertal.de}

\address[BUW]{University of Wuppertal, 
Applied and Computational Mathematics,\\
Gaußstrasse 20, 42119 Wuppertal, Germany}

\address[BUWO]{University of Wuppertal, 
Optimization Group,\\
Gaußstrasse 20, 42119 Wuppertal, Germany}


\begin{abstract}
We present a data-driven approach to identifying linear index-1 differential-algebraic pH systems (pH-DAEs) based on input-output measurements. In comparison to the identification of port-Hamiltonian (pH) systems, the algebraic constraint and the index condition pose additional challenges.
First, we establish a structure-preserving formulation of the considered pH-DAE class and derive an implicit midpoint discretization that preserves the algebraic constraints and discrete dissipation inequality. 
We formulate the identification problem as a regularized least-squares minimization problem subject to the pH-DAE dynamics. 
Exploiting the index-1 structure, we reduce the constrained problem to an unconstrained optimization problem over the system parameters while preserving the port-Hamiltonian structure. 
Next, we derive an adjoint-based formulation to efficiently evaluate the gradient of the resulting reduced cost functional. 
This enables us to use gradient-based optimization methods for parameter estimation. Under suitable assumptions on the admissible parameter set, the existence of a minimizer is established. 
Numerical experiments demonstrate that the proposed approach can identify surrogate pH-DAE systems that accurately reproduce the input-output behavior of reference systems. 
Further investigations 
show the approach's potential for identifying reduced-order surrogate models. Cross-validation with independent input signals confirms the predictive capability of the identified models.

\end{abstract}

\begin{keyword}
Differential-algebraic equations \sep port-Hamiltonian systems \sep system identification \sep data-driven modelling \sep  structure-preserving methods \sep adjoint methods 

\textit{2020 Mathematics Subject Classification:} 93B30, 34A09, 65L80, 90C30
\end{keyword}


\journal{Computers & Mathematics with Applications} 
\end{frontmatter}

\section{Introduction}\label{sec:introduction}

Mathematical models of physical systems are often characterized by conservation laws, dissipation mechanisms, and interconnections between different physical subsystems. 
Preserving these structural properties in the mathematical representation is important not only for the physical interpretability of the model, but also for the stability and robustness of numerical simulations. 
\textit{Port-Hamiltonian (pH) systems} provide a systematic framework for structure-preserving modelling based on the balance of energy and power. 
In particular, the pH formulation separates the power-conserving part of the dynamics from dissipative effects and external interactions.
This structure is preserved under power-conserving interconnections and therefore provides a natural framework for the modelling and coupling of complex multiphysics systems; see, e.g., \cite{schaftPortHamiltonianSystemsNetwork2004, jacobLinearPortHamiltonianSystems2012}.

In many applications, however, 
a detailed physical model of a subsystem is not available. 
This may be caused by incomplete knowledge of the underlying 
physical parameters, inaccessible internal variables, or the use of proprietary simulation software for which only input-output information is available. 
In such situations, data-driven system identification provides a means of constructing surrogate models from measurements or simulation data. 
For  port-Hamiltonian systems, it is  
desirable that the identified model retains the pH structure, since otherwise important properties of the original physical system may be lost during the identification process.

Several approaches for data-driven identification of pH systems have been developed in recent years. Frequency-domain approaches based on the Loewner framework can directly construct pH realizations from suitable frequency response data \cite{BennerGoyalVanDooren2020}.
A related non-intrusive approach for time-domain data first extracts frequency-response information from input-output measurements and subsequently constructs a pH realization using a Loewner-based procedure \cite{CherifiGoyalBenner2022}. 
These approaches are attractive when frequency-response data are available or can be reliably inferred from the measurements. 
An alternative strategy is to formulate the identification directly as an optimization problem. 
In \cite{guntherDatadrivenAdjointbasedCalibration2024}, an adjoint-based time-domain calibration method was developed for linear pH ordinary differential equations (ODEs). 
The system matrices are determined directly from input-output data by minimizing a data-misfit functional subject to the pH dynamics. 
Importantly, the pH constraints are incorporated directly into the optimization problem, so that the identified systems retain the skew-symmetry and positive-(semi-)definiteness properties required by the pH formulation.
Sensitivity-based identification methods have also been considered for linear pH systems \cite{guntherStructurepreservingIdentificationPortHamiltonian2023}.

Differential-algebraic equations (DAEs) arise naturally when dynamics have to obey strict structural constraints, for example in mechanical multibody systems such as robotics. 
Linear pH descriptor systems were systematically developed in \cite{beattieLinearPortHamiltonianDescriptor2018}, where the pH structure was extended to constrained dynamical systems of arbitrary differentiation index. 
The authors showed how regularization procedures can be formulated while preserving the pH structure. 
Numerical methods for DAEs with pH applications have 
been developed, including structure-preserving splitting techniques; see, e.g., \cite{bartelOperatorSplittingBased2023, bartelSplittingTechniquesDAEs2025}.

The data-driven identification of pH-DAE systems, however, is considerably less developed than the 
problem for pH-ODE systems. 
One recent approach was proposed by Zaspel and Günther \cite{zaspelDatadrivenIdentificationPortHamiltonian2024}, where Gaussian processes are employed to identify nonlinear effort functions of pH-DAE systems from input and state data.
The approach was demonstrated for index-1 and index-three pH-DAE systems arising in network design and constrained multibody dynamics. 
While this provides a data-driven framework for pH-DAEs, it requires state information and follows a fundamentally different, nonparametric modelling strategy.

Here, we consider the complementary problem of identifying \textit{linear index-1 pH-DAE systems directly from input-output measurements}. 
Our approach extends the adjoint-based time-domain calibration framework for pH-ODE systems in \cite{guntherDatadrivenAdjointbasedCalibration2024} to the differential-algebraic setting. 
The main difficulty is that the system dynamics are subject to algebraic constraints and that the differential and algebraic variables must remain compatible with the index-1 condition throughout the identification procedure. 
We therefore develop a structure-preserving numerical treatment of the underlying pH-DAE and exploit its index-1 structure to formulate the identification problem as an optimization problem in the system parameters.

More precisely, we first formulate the considered class of linear pH-DAE systems and establish the structural properties relevant for the identification problem. 
For index-1 systems, we derive a numerical discretization that respects the algebraic constraints and the dissipative structure. 
The resulting data-driven identification problem is posed as a least-squares minimization problem subject to the pH-DAE dynamics. 
By exploiting the index-1 structure, the constrained problem can be reduced to an unconstrained optimization problem over an appropriately parameterized set of pH system matrices. 
We derive the gradient of the reduced cost functional using an adjoint formulation. 
This avoids the explicit computation of sensitivities with respect to every individual system parameter and provides an efficient basis for gradient-based optimization. 
Under suitable assumptions, we additionally establish the existence of a minimizer of the identification problem.

The proposed approach is assessed numerically using several reference systems. 
The experiments investigate the influence of the number of differential variables used in the identified model and examine whether lower-dimensional pH-DAE surrogate models can reproduce the observed input-output behaviour. 
In addition, the identified models are evaluated using independent test signals. 
The numerical results demonstrate that the identified systems can reproduce the input-output behaviour of the reference systems while retaining the port-Hamiltonian structure. 
In particular, the cross-validation experiments indicate that the identified models are not merely fitted to the training signal but can provide useful surrogate models for unseen inputs.

The paper is organized as follows. 
Section~\ref{sec:pH} introduces the port-Hamiltonian framework and recalls the properties needed in the subsequent analysis. 
Section~\ref{sec:DAE} discusses differential-algebraic systems and introduces the considered class of linear port-Hamiltonian DAEs, with particular emphasis on the index-1 case.
The data-driven identification problem and the existence result are formulated in Section~\ref{sec:identification}.
Section~\ref{sec:optimization} derives the parameterization and the adjoint-based gradient used for the numerical
optimization. 
Numerical experiments are presented in Section~\ref{sec:numerical}, including the model-reduction and
cross-validation studies.
Finally, Section~\ref{sec:conclusion} summarizes the main results and discusses possible directions for future work.

\section{Port-Hamiltonian Systems}\label{sec:pH}
We consider the Hamiltonian equations for a mechanical system 
\cite{schaftPortHamiltonianSystemsNetwork2004, bartelPortHamiltonianSystemsModelling2023}
\begin{equation}\label{eq.HamiltonianEquations}
    \dot{q} = \frac{\partial H(p,q)}{\partial p},  \qquad
    \dot{p} = -\frac{\partial H(p,q)}{\partial q} + F,
\end{equation}
with the generalized coordinates $q=(q_1,\ldots,q_k)^\top$, the generalized momenta $p=(p_1,\ldots,p_k)^\top$ and the vector of generalized forces $F$.
The function $H(p,q)$ is called
\textit{separable Hamiltonian} and is given by the \textit{total energy} 
$H(p,q)=U(p)+V(q)$, i.e., the sum of kinetic energy $U$ and potential energy $V$.
An important  property of such systems is the \textit{energy balance}
\begin{equation}
    \frac{\dd}{\dd t} H = \frac{\partial^\top H}{\partial q}(q,p) \dot{q} + 
    \frac{\partial^\top H}{\partial p}(q,p) \dot{p} = \frac{\partial^\top H}{\partial p}(q,p) F = \dot{q}^\top F.
\end{equation}

The Hamiltonian equations \eqref{eq.HamiltonianEquations} without external force $F$ can be expressed in matrix form 
\begin{align}
    \dot{x} = J \nabla H(x)\quad\text{with}\quad
    x = (q,p)^\top\quad\text{and}\quad J =\begin{psmallmatrix} 0 & I\\ -I & 0 \end{psmallmatrix}.
\end{align}
Next, we replace the matrix $J$ by an arbitrary skew-symmetric matrix that is again called $J$, and add a positive semi-definite matrix $R$ to include a dissipative term (energy loss) 
\begin{equation}
    \dot{x} = (J-R) \nabla H(x).
\end{equation} 
Finally, we include an input and an output to couple the system to the environment.
\begin{definition}[\cite{lorenzOperatorSplittingCoupled2025b}]\label{def:PHS}
    Let $H\colon\R^n\to\R$ be a twice continuously differentiable function, $J\in \R^{n\times n}$ a skew-symmetric matrix, i.e., $J^\top=-J$, and $R\in \R^{n\times n}$ a symmetric positive semi-definite matrix, i.e., $R \succeq 0$. The system
    \begin{subequations}\label{eq:PHS}
        \begin{align}
            \dot{x}(t) &= (J-R) \nabla H(x(t)) + B u(t),  \quad x(0) = x_0, \label{eq:ph-ode}\\
            y &= B^\top \nabla H(x(t)),
        \end{align}
    \end{subequations}
    is called a \textit{port-Hamiltonian system} (pH system) where $u\colon[0,T]\to \R^m$ and $y\colon[0,T]\to\R^m$ denote an input and output signal with $T>0$ the time horizon. 
    The matrix $B\in \R^{n\times m}$ is called \textit{port matrix}. 
\end{definition}
In the linear case, the Hamiltonian $H$ is a quadratic function, i.e., $H(x)=\frac{1}{2}x^\top Q x$, where $Q\in \R^{n\times n}$. 
Since $H$ represents the energy of the system, we require that $H\ge 0$. 
Therefore, $Q$ is a symmetric positive semi-definite matrix. 

A fundamental property of port-Hamiltonian systems is the \textit{dissipation inequality}, i.e.\
for the Hamiltonian $H$ and the solutions $x(t)$ and $y(t)$ of the system \eqref{eq:PHS} we have
 \begin{equation}\label{eq:diss_ineq}
        \frac{\dd}{\dd t} H(x(t)) \le y(t)^\top u(t),
    \end{equation}
or in integral form 
\begin{equation}
    H(x(t+h)) -  H(x(t)) \le \int_t^{t+h} y(\tau)^\top u(\tau) \,d\tau.
\end{equation} 
Finally, the second important property is the \textit{interconnection}, i.e.\ the proper coupling of pH systems leads again to a monolithic pH system, see \cite{ehrhardtStructurepreservingCouplingDecoupling2026} for details.

The next step is to design a numerical scheme that preserves the dissipation inequality \eqref{eq:diss_ineq}.
For simplicity of notation, we consider a uniform grid for discretization: $0 < t_1 < \dots < t_{\rm end}$ with step size $h = t_{n+1} - t_n>0$ and numerical solutions $x_n\approx x(t_n)$. 
A class of methods that satisfied the
discrete version of the dissipation inequality \eqref{eq:diss_ineq}
\begin{equation}
    \frac{H(x_{n+1})-H(x_n)}{h}\le y_{n+1}^\top u_{n+1}
\end{equation}
is the discrete gradient method.
\begin{definition}[\cite{gonzalezTimeIntegrationDiscrete1996}]\label{def:dis_grad}
    The discrete gradient of a continuously differentiable function $H\colon \R^n \to \R$ is the mapping
    $\overline{\nabla} H\colon\R^n \times \R^n \to{} \R$ with the two properties
    \begin{align}
        \label{eq:dis_grad1}
        \overline{\nabla} H(y,x)^\top(y-x) &= H(y) - H(x) \quad \text{for all }x,y \in \R^n,\\ \label{eq:dis_grad2}
        \overline{\nabla} H(x,x) &= \nabla H(x) \quad \text{for all } x \in \R^n.
    \end{align}
\end{definition}

\begin{definition}
Let $\overline{\nabla} H$ be the discrete gradient of the function $H$ defined in Definition \ref{def:dis_grad}. Furthermore, let $J$ be a skew-symmetric matrix, $R$ a positive semi-definite matrix. 
Then
\begin{align}\label{def:dis_meth} 
    x_{n+1} = x_n + h (J-R) \overline{\nabla} H(x_{n+1},x_n) + h Bu_{n+1}
\end{align} defines the \textit{discrete gradient method}. 
The definition of $u_{n+1}$ depends on the choice of $\bar{\nabla}H$.
\end{definition}

For a quadratic Hamiltonian
$H(x)=\frac{1}{2}x^\top Qx$ with $Q=Q^\top$, 
the \textit{(midpoint) discrete gradient} \cite{gonzalezTimeIntegrationDiscrete1996} reduces to the midpoint evaluation of the gradient,
\begin{equation}\label{eq:mittelp}
    \overline{\nabla} H(x_{n+1},x_n)
    = \nabla H\Bigl(\frac{x_n+x_{n+1}}{2}\Bigr)
    = Q\Bigl(\frac{x_n+x_{n+1}}{2}\Bigr).
\end{equation}
Consequently, the resulting discrete-gradient scheme coincides with
the implicit midpoint rule; see, e.g., \cite[Section~3.1]{HairerLubich2014}.
\section{Port-Hamiltonian Differential-Algebraic Systems}\label{sec:DAE}
In this section we consider the special case of linear port-Hamiltonian DAE systems which represent a generalized form systems of port-Hamiltonian ODEs. 
First, we introduce port-Hamiltonian DAE systems and subsequently discuss the properties of such systems.
\begin{definition}[pH-DAE,\cite{bartelSplittingTechniquesDAEs2025}] \label{def:pH-DAE}
    A linear constant coefficient DAE system of the form
    \begin{subequations}\label{eq:pH-DAE_general}
    \begin{align}
        E\dot{x}&=(J-R)Qx+B u,\\
        y&=B^\top Qx,
    \end{align}
    \end{subequations}
    with $E$, $Q$, $J$, $R\in\mathbb{R}^{n\times n}$, $B\in\mathbb{R}^{n\times m}$, and $x = x(t) \in \mathbb{R}^n$, $u= u(t)$, $y = y(t)\in\mathbb{R}^m$ with $t \in [0,T]$, 
    is called a \textit{port-Hamiltonian differential-algebraic system} (pH-DAE), if
%
\begin{enumerate}

\item the differential-algebraic operator
\begin{equation*}
   Q^\top E\frac{\dd}{\dd t}-Q^\top JQ \colon \mathcal{X}\subset \mathcal{C}^1([0,T],\R^n)\to \mathcal{C}^0([0,T],\mathbb{R}^n)
\end{equation*}
is skew-adjoint, i.e.,
$Q^\top J^\top Q=-Q^\top JQ$ and $Q^\top E=E^\top Q$ hold,

\item the product $Q^\top E$ is positive semi-definite, i.e., $Q^\top E \succeq 0$, 

\item the passivity matrix $W=\begin{pmatrix}Q^\top RQ\end{pmatrix}$
is symmetric positive semi-definite, $W \succeq 0$.
\end{enumerate}
The underlying quadratic Hamiltonian $\mH\colon\mathbb{R}^n\to\mathbb{R}$ of the system is
$\mH (x)=\frac{1}{2}x^\top Q^\top Ex$.
\end{definition}

In this work we focus on the identification of pH-DAE systems in the form
\begin{align}\label{eq:eq14}
    \begin{pmatrix}
        I_r & 0\\
        0 & 0
    \end{pmatrix}
    \begin{pmatrix}
        \dot{x}_1 \\ \dot{x}_2
    \end{pmatrix}
    = \left(
    \begin{pmatrix}
        J_{11} & J_{12} \\
        J_{21}  & J_{22} 
    \end{pmatrix}
    -
    \begin{pmatrix}
        R_{11}   & R_{12} \\
        R_{21}   & R_{22} 
    \end{pmatrix}
    \right)
    \begin{pmatrix}
        x_1 \\ x_2
    \end{pmatrix}
    + \begin{pmatrix}
        B_1 \\ B_2
    \end{pmatrix} 
    u.
\end{align}
We note that, provided $Q$ is regular, any linear pH-DAE system can be transformed into a system \eqref{eq:eq14},
which is already given in semi-explicit form. Componentwise we obtain
\begin{align}
    \dot{x}_1 &= (J_{11}-R_{11}) x_1 + (J_{12}-R_{12}) x_2 + B_1 u, \label{eq:diffx1}\\
    0 &= (J_{21}-R_{21}) x_1 + (J_{22}-R_{22}) x_2 + B_2 u. \label{eq:algebraic_part}
\end{align}
The first $r$ components of vector $x$ correspond to the differential variable $x_1$, and the remaining $n-r$ components correspond to the algebraic variable $x_2$.
Under the index~1 assumption that the matrix $(J_{22}-R_{22})$ is regular, the algebraic equation can be solved for $x_2$:
\begin{equation}\label{eq:x_2}
        x_2 = -(J_{22}-R_{22})^{-1}\bigl((J_{21}-R_{21})x_1 + B_2 u\bigr).
\end{equation}
If we substitute the algebraic part \eqref{eq:x_2} into the differential part \eqref{eq:diffx1}, we obtain 
\begin{equation}\label{eq:x_1}
\begin{split}
        \dot{x}_1 
         &= \bigl( J_{11}-R_{11} - (J_{12}-R_{12})(J_{22}-R_{22})^{-1}(J_{21}-R_{21}) \bigr)x_1 \\
         & \qquad + \bigl(B_1 - (J_{12}-R_{12})(J_{22}-R_{22})^{-1} B_2 \bigr) u,
\end{split}
\end{equation}
which is a closed-form expression for $x_1$.
Hence, we obtain a simple ODE system for $x_1$ \eqref{eq:x_1} and an explicit equation for $x_2$ \eqref{eq:x_2}. 
The initial condition for $x_1(0)=x_{1,0}$ can be chosen arbitrarily, while the initial condition for the algebraic variable $x_2$ is given by
\begin{equation}\label{eq:x_2_initial}
    x_{2,0} \coloneq x_2(0) = -(J_{22}-R_{22})^{-1}\bigl((J_{21}-R_{21})x_1(0) + B_2 u(0)\bigr).
\end{equation} 
Thus, instead of writing $x(0)=x_0\in \R^n$ as a consistent initial value, we can write $x_1(0)=x_{1,0}\in \R^r$ and determine $x_2(0)=x_{2,0}\in \R^{n-r}$ using \eqref{eq:x_2_initial}. 
\begin{remark}
    In \cite{beattieLinearPortHamiltonianDescriptor2018}, it has been shown that the index of the pH-DAE \eqref{eq:pH-DAE_general} is at most two. 
    Furthermore, it was shown that, under certain assumptions an index~2 pH-DAE can be transformed into an equivalent index~1 pH-DAE using a structure preserving regularization. 
\end{remark}

In the sequel we will consider again the implicit midpoint and derive the method directly for pH-DAE systems by approximating the derivative $\dot{x}(t_{n+1/2})$ with the finite difference $\frac{x_{n+1} - x_n}{h}$ and $x(t_{n+1/2})$ with the midpoint $\frac{x_n + x_{n+1}}{2}$. This yields
\begin{align}\label{eq:DAE_midpoint_rule}
    E\,\frac{x_{n+1} - x_n}{h} = (J-R)\, \frac{x_n+x_{n+1}}{2} + B \,\frac{u_n + u_{n+1}}{2}.
\end{align}
In the sequel, we restrict ourselves to the transformed pH-DAE systems of form \eqref{eq:eq14}, i.e.\
\begin{subequations}\label{eq:pH-DAE_transformed}
        \begin{align}
        E\,\dot{x} &= (J-R)\,x + B\,u,\quad\text{with}\quad
        E= \diag(I_r,0),\\
        y &= B^\top x. 
        \end{align}
    \end{subequations}
\begin{definition}\label{def:midpoint_rule}
    The implicit midpoint rule for solving the system \eqref{eq:pH-DAE_transformed} is given by
    \begin{equation}\label{eq:midpoint_rule}
        \begin{split}
            \Bigl(E-\frac{h}{2}(J-R)\Bigr) x_{n+1} &=  \Bigl(E+\frac{h}{2}(J-R) \Bigr) x_n + \frac{h}{2} B (u_n + u_{n+1}), \\
            x_{n+1} &= \Bigl(E- \frac{h}{2}(J-R)\Bigr)^{-1} \biggl[\Bigl(E+ \frac{h}{2}(J-R)\Bigr) x_n +  \frac{h}{2}B (u_n + u_{n+1})\biggr].
        \end{split}
    \end{equation}
\end{definition}
It was shown in \cite{bartelSplittingTechniquesDAEs2025}
that the rule \eqref{eq:midpoint_rule} is well-defined for index-1 systems and satisfies
    the algebraic constraints and the discrete form of the dissipation inequality, i.e.\ the midpoint rule is structure-preserving for pH-DAE systems under certain conditions. 

\section{Identification Problem}\label{sec:identification}
In the previous section we introduced pH-DAE systems. 
Now, we will investigate how we can reconstruct such systems from measured data $(t_i,u_i,y_i)_{i=0,\dots,N}$. 
First, we will formulate the minimization problem.
We assume that we have an input signal $u\colon [0,T] \to \R^m$ and the corresponding output signal $\yd \colon [0,T]\to \R^n$ from an unknown system. If only discrete values $(t_i,u_i,y_i)_{i=0,\dots,N}$ are given, the functions $u(t)$ and $y(t)$ are obtained by interpolation.
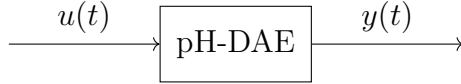
\begin{figure}[H]
  \centering
  
    \begin{tikzpicture}
  \node[draw, minimum width=2cm, minimum height=1cm] (system) {pH-DAE};

  \draw[->] (-3,0) -- (system.west) node[midway, above] {$u(t)$};

  \draw[->] (system.east) -- ++(2,0) node[midway, above] {$y(t)$};

\end{tikzpicture}
\caption{A linear pH-DAE system to be determined.}
\end{figure}
The main assumption in the identification process is that the data are generated by a linear index-1 pH-DAE system. 
We consider the calibration problem with \textit{cost functional}
\begin{equation}\label{eq:costfunctional}
    \J(y,w) = \frac{1}{2} \int_0^T |y(t) - \yd(t) |^2 \,\dd t + \frac{\lambda}{2} | w-w_{\rm ref} |^2, \qquad \lambda \ge 0,
\end{equation}
subject to the state constraint or \textit{state equation} (cf.\ \eqref{eq:pH-DAE_transformed}):
\begin{subequations}\label{eq:state}
	\begin{align}
        E\dot{x} &= (J-R)x + Bu, \qquad   x_1(0)=x_{1,0} \in \R^r, \\
        y &= B^\top x,
    \end{align}
\end{subequations}
where $w = (J,R,B,x_{1,0})$ contains the matrices and the initial data to be identified by minimizing $\J$. 
If applicable, $w_{\rm ref}$ contains reference data. Otherwise, we set $\lambda = 0$.
%

Since we assume that we are investigating a pH-DAE, the following properties hold:
\begin{equation}\label{eq:matrix}
 J,R \in \R^{n\times n}, \qquad J^\top = -J, \qquad R \succeq 0, \qquad B \in \R^{n\times m}, \qquad 
 (J_{22}-R_{22}) \text{ is regular}.
\end{equation}
%
However, the proof of continuous dependence on the data shows that the index-1 condition on $J_{22}-R_{22}$ to be regular, alone is insufficient. 
Therefore, we additionally assume that the norm of the inverse matrix is bounded, i.e., there exists $C\gg 0$ such that is holds $\|(J_{22}-R_{22})^{-1}\| \le C$.
Note that this is equivalent to the smallest singular value being greater than or equal to some threshold $\epsilon>0$, i.e., $\sigma_{\min}(J_{22}-R_{22})\ge \epsilon >0$, cf.\ \cite{Golub2013}.
In other words, the matrix remains sufficiently far away from a singularity. 
%
Therefore, we define the \textit{admissible parameter set} as
\begin{equation}
\Uad:=\bigl\{w=(J,R,B,x_{1,0})\in W\colon J^\top=-J,\quad
R\succeq0,\quad
\sigma_{\min}(J_{22}-R_{22})\ge\epsilon>0 \bigr\},
\end{equation}
where 
$W = \R^{n\times n} \times \R^{n\times n} \times \R^{n\times m} \times \R^r$
denotes the parameter space. 
Since the smallest singular value depends continuously on the matrix entries, $\Uad$ is closed.
Moreover, the condition $\sigma_{\min}(J_{22}-R_{22})\ge\epsilon$ ensures the uniform bound
\begin{equation*}
   \|(J_{22}-R_{22})^{-1}\|_2 \le \epsilon^{-1}.
\end{equation*}

We note that the solution to the identification process is also in $\Uad$, since we restrict ourselves to index-1 pH-DAEs. Moreover, we define the output space $Y=H^1((0,T),\R^m)=\{y \in L^2((0,T),\R^m)\colon \dot{y}\in L^2((0,T),\R^m)\}$
equipped with the norm 
\begin{equation}
    \|y\|_{H^1((0,T),\R^m)} = \bigl(\|y\|_{L^2((0,T),\R^m)}^2 + \|\dot{y}\|_{L^2((0,T),\R^m)}^2\bigr)^{1/2}.
\end{equation}
This choice of the space can be justified by the fact that $y$ and $x$ have the same structure, since $y=B^\top x$, and $x$ must be differentiable, i.e., $x\in H^1((0,T),\R^n)$.
In addition, we are introducing the following notation
\begin{equation*}
  \J_1(y) := \frac{1}{2} \int_0^T |y(t) - \yd(t) |^2 \,\dd   t, \qquad 
  \J_2(w) := \frac{\lambda}{2} | w-w_{\rm ref} |^2.
\end{equation*}
If no information on reference data are available, we recommend to set $\lambda=0$. In this case, the cost function reduces to $\J_1(y)$.

Now we have all important ingredients in order to formulate our calibration task. 
\begin{problem}\label{problem}
    Find system matrices and the initial data, $w=(J,R,B,x_{1,0})$, which solve 
    \begin{equation} \tag{IdP}\label{eq:problem}
    \min\limits_{(y,w) \in Y \times \Uad } \J(y,w) \quad \text{subject to} \quad \eqref{eq:state}.
    \end{equation}
\end{problem} 
This minimization problem (IdP) is solved using an iterative method and will be discussed in Section~\ref{sec:optimization}. 
The idea is to start from an initial guess $w^0$ and solve the state equation to obtain state variable $x$ and output $y$. 
Based on this, we evaluate $\J$ and calculate the gradient which is then used for a gradient descent to minimize $\J$. 
The aim now is to prove that the \textit{identification problem} \eqref{eq:problem} is well-posed and solvable under certain conditions. To this end, we first make some observations and then prove that the solution to the state equation depends continuously on the data.
\begin{remark}
    The identification problem \eqref{eq:problem} still has two free parameters.
    While $m$ is determined by the dimension of $\yd$, the dimension of the internal state $n$ and the size of the algebraic part $n-r$ are unknown. We therefore assume that $n$ has already been specified. 
    For a deliberately small choice of $n$, this can be interpreted as a model order reduction for the internal state. 
    To determine the number of differential variables, it is necessary to solve the minimization problem for various $r$. 
    Then, one the hand, we can compare the final costs values, and on the other hand, examine the behavior in response to a second test input $u_{\rm test}$ (cross-validation).  
\end{remark}

\begin{lemma}\label{lemma1}
    For every $w\in \Uad$ and $u\in C^1([0,T],\R^m)$, the state equation \eqref{eq:state} admits a unique solution $x\in C^1([0,T],\R^n)$. Moreover, the solution depends continuously on the data. More precisely, for $w,\bar{w}\in\Uad$ and the corresponding solutions $x,\bar{x}$, there exists a constant $C>0$ such that
    \begin{equation*}
        \| x - \bar{x} \|_{H^1((0,T),\R^n)} \le C \|w-\bar{w}\|_{\Uad}.
    \end{equation*}
\end{lemma}
\begin{proof}
    For given $w\in\Uad$ and $u\in C([0,T],\R^m)$ the existence of a unique solution to \eqref{eq:state} can be easily shown.
    For the second statement we estimate $\|x-\bar{x}\|_{L_2((0,T),\R^n)}$ and $\bigl\|\frac{\dd}{\dd t}x- \frac{\dd}{\dd t}\bar{x}\bigr\|_{L_2((0,T),\R^n)}$ separately. \\
    We consider $x$ again in terms of its differential and algebraic parts and obtain
    \begin{equation}\label{eq:lemma1}
        |x(t) - \bar{x}(t)|^2 = 
        |x_1(t) - \bar{x}_1(t)|^2 + |x_2(t) - \bar{x}_2(t)|^2,
    \end{equation}
    with 
    \begin{equation*}
    \begin{split}
       \dot{x}_1 &=\bigl( J_{11}-R_{11} - (J_{12}-R_{12})(J_{22}-R_{22})^{-1}(J_{21}-R_{21}) \bigr)x_1 \\
         & \qquad + \bigl(B_1 - (J_{12}-R_{12})(J_{22}-R_{22})^{-1} B_2 \bigr)u\\
         &=: Mx_1 + N u,\\
        x_2 &= -(J_{22}-R_{22})^{-1}(J_{21}-R_{21})x_1 -(J_{22}-R_{22})^{-1} B_2 u\\
        &=: K x_1 + Lu,
    \end{split}
     \end{equation*}
    according to equations \eqref{eq:x_1} and \eqref{eq:x_2}. We define $\bar M, \bar N, \bar K$ and $\bar L$ analogous for the solution corresponding to $\bar w$. The integral representation of the ODE then yields
    \begin{equation*}
        |x_1(t)-\bar{x}_1(t)|^2 \le 2 |x_{1,0}-\bar{x}_{1,0}|^2 + 2T \int_0^t |Mx_1(s) - \bar{M}\bar{x}_1(s) +(N-\bar{N})u(s)|^2 \dd s.
    \end{equation*}
    To estimate the integral, we add the term $\bar{M}x_1 - \bar{M}x_1$ to the integrant and obtain
    \begin{equation*}
    \begin{split}
        |M_1 x_1 -\bar{M}\bar{x}_1 &- \bar{M}x_1 + \bar{M}x_1 +(N-\bar{N})|^2 \\
        &= |(M-\bar{M})x_1 + \bar{M}(x_1-\bar{x}_1)+(N-\bar{N})u|^2  \\
        &\le  3 \|M-\bar{M}\|^2 |x_1|^2 + 3\|\bar{M}\|^2 |x_1-\bar{x}_1|^2+ 3\|N-\bar{N}\|^2 |u|^2.
    \end{split}
    \end{equation*}
    Inserting into the integral yields
    \begin{equation*}
    \begin{split}
        |x_1(t)-\bar{x}_1(t)|^2 &\le 2 |x_{1,0}-\bar{x}_{1,0}|^2 + 6T \int_{0}^{t} \|M-\bar{M}\|^2 |x_1|^2 \dd s + 6T \int_{0}^{t}\|\bar{M}\|^2 |x_1-\bar{x}_1|^2 \dd s
        \\ &\qquad \qquad \qquad +6T \int_{0}^{t}\|N-\bar{N}\|^2 |u|^2 \dd s\\
        &\le C_2 |w-\bar{w}|^2 +  \int_{0}^{t} C_3|x_1-\bar{x}_1|^2 \dd s.
    \end{split}
    \end{equation*}
    
    An application of Gronwall inequality yields
    \begin{equation}\label{eq:lemma_proof2}
        |x_1(t)-\bar{x}_1(t)|^2 \le C_4 |w-\bar{w}|^2 .
    \end{equation}
    For the algebraic part $x_2$, we obtain the following estimate
     \begin{equation*}
    \begin{split}
        |x_2(t) - \bar{x_2}(t)|^2  
        &= |(K - \bar{K})x_1 + \bar{K}(x_1 -\bar{x}_1 )+ (L-\bar{L})u|^2 \\
        &\le 3\|K-\bar{K}\|^2 |x_1|^2 + 3\|\bar{K}\|^2 |x_1 -\bar{x}_1|^2 + 3\|L-\bar{L}\|^2 |u|^2\\
        &\le C_5 |w-\bar{w}|^2 .
    \end{split}
    \end{equation*}
    In the final step, we used the estimate \eqref{eq:lemma_proof2} 
    to estimate the norm of $K-\bar{K}$ and $L-\bar{L}$. 
    
    By an integration over $[0,T]$ and using equation \eqref{lemma1}, we obtain
    \begin{equation}\label{eq:lemma_proof3}
        \|x-\bar{x}\|_{L^2((0,T),\R^n)} \le C_6 |w-\bar{w}|.
    \end{equation}
    The second part of the proof consists analogously of estimate the $L^2$-difference of the derivatives.
    For the differential part, we obtain 
    \begin{equation*}
        |\frac{\dd}{\dd t}(x_1(t)-\bar{x}_1(t))|^2 = C_7 |w-\bar{w}|^2  + \int_0^t C_8|\frac{\dd}{\dd s}(x_1-\bar{x}_1)|^2 \dd s.
    \end{equation*}
    For the algebraic part, we obtain 
     \begin{equation*}
    \begin{split}
        |\frac{\dd}{\dd t}(x_2(t)-\bar{x}_2(t))|^2 
        &\le 3\|K-\bar{K}\|^2 |\frac{\dd}{\dd t}x_1|^2 + 3\|\bar{K}\|^2 \frac{\dd}{\dd t}|x_1 -\bar{x}_1|^2 + 3\|L-\bar{L}\|^2 \frac{\dd}{\dd t}|u(t)|^2\\
        &\le C_9 |w-\bar{w}|^2 .
         \end{split}
    \end{equation*}
    Note $C_9$ depends on $\max\limits_{t\in[0,T]} \frac{d}{dt} |u(t)|^2 $, which is finite, since $u \in C^1([0,T],\R^m)$. 
    Again, Gronwall lemma and integration of both terms yields
    \begin{equation}\label{eq:lemma_proof4}
        \Bigl\|\frac{\dd}{\dd t}(x-x')\Bigr\|_{L^2((0,T),\R^n)} \le C_{10} |w-w'|^2.
    \end{equation}
    Adding \eqref{eq:lemma_proof3} and \eqref{eq:lemma_proof4} and taking square root yields the desired result.
\end{proof}
Thus, the state equation \eqref{eq:state} is well-defined, and we can introduce the \textit{solution operator}.
\begin{definition}
    The \textit{solution operator} $S$ is defined as 
    \begin{equation*}
       S\colon\Uad\to H^1((0,T),\R^n), \qquad S(w) = x.
     \end{equation*}
\end{definition}
Using this solution operator $S(w)$, we define the \textit{reduced cost functional}
\begin{equation}\label{eq:reduced}
            \hat{\J}(w) 
            = \frac{1}{2} \int_0^T | B^\top S(w)(t) - \yd(t) |^2 \dd t + \frac{\lambda}{2} |w-w_{\rm ref}|^2.
\end{equation}

Thus, the minimization Problem~\ref{problem} reduces to
\begin{equation*} 
   \min \hat{\J}(w), \qquad w\in \Uad, 
\end{equation*}
and the DAE constraint \eqref{eq:state} is only handled implicitly. More specifically, to evaluate the cost functional $\hat{\J}$, the state constraint is first solved, followed by the calculation of the gradient with $y$. The gradient-descent method used for the numerical results is based on the reduced problem formulation involving $\hat{\J}$.

Next, we define the \textit{state operator} and then show that it is weakly continuous.
\begin{definition}
    The mapping 
    \begin{gather}\label{eq:e}
    e \colon H^1((0,T),\R^n) \times \Uad \to L^2((0,T),\R^n) \times \R^r, \\
    \quad (x,w) \mapsto 
    \begin{pmatrix} E\frac{\dd}{\dd t} x - (J-R)x - Bu \\
               x_1(0) - x_{1,0}
    \end{pmatrix}
    \end{gather}
    with $w=(J,R,B,x_{1,0})$ and $E =\diag(I_r,0) $ 
    defines the state operator.
\end{definition}
Note that the state operator allows us to rewrite our optimization problem as 
\begin{equation*}
    \min_{x\in H((0,T),\R^n), \; w \in \Uad} J(x,w) \quad \text{ subject to } \quad e(x,w) = 0.
\end{equation*} 
This formulation of our identification problem will be used later for the Lagrange approach.
\begin{lemma}\label{lemma2}
    The state operator $e$, defined in \eqref{eq:e}, is weakly continuous.
\end{lemma}
\begin{proof}

    We consider $\{x_n\}_n \subset H^1((0,T),\R^n)$ with $x_n \rightharpoonup x$ and $\{w_n\}_n \subset \Uad$ with $w_n \rightharpoonup w$.
    Note that $\Uad$ is finite dimensional, thus weak and strong convergence coincide. Then for $\varphi\in L^2((0,T),\R^n)$ we obtain
    \begin{equation*}
    \begin{split}
        &\int_0^T \Bigl( E\frac{\dd}{\dd t} x_n - (J_n-R_n)x_n(t) - B_n u(t)\Bigr) \cdot \, \varphi(t) \, \dd t \\
        & \qquad - \int_0^T \Bigl( E\frac{\dd}{\dd t} x
        - (J-R) x(t) - B u(t)\Bigr) \cdot\varphi(t)\, \dd t  \\
        &= \int_0^T \biggl(\Bigl( E\frac{\dd}{\dd t} x_n - E\frac{\dd}{\dd t}x \Bigr) + (B-B_n)u(t)\biggr) \cdot \, \varphi(t)\, \dd t \\ 
        &+ \int_0^T \Big[Jx(t) -J_n x_n(t) + R_n x_n(t) -Rx(t) +\underbrace{Jx_n(t)-Jx_n(t) + Rx_n(t)-Rx_n(t)}_{=0} \Big] \, \cdot\varphi(t) \, \dd t \\  
        &= \int_0^T \biggl(\Bigl( E\frac{\dd}{\dd t} x_n - E\frac{\dd}{\dd t}x \Bigr) + (B-B_n)u(t)\biggr) \cdot \, \varphi(t)\, \dd t \\ 
        &\qquad + \int_0^T
        \Big[
        ((J-J_n)-(R-R_n))x_n(t) \, \cdot\varphi(t) + (J-R)(x(t)-x_n(t)) \, \cdot\varphi(t) \Big] \dd t.
    \end{split}
  \end{equation*}
  
    By the weak convergence of $x_n$ and $w_n$  we conclude that all terms tend to zero as $n\to \infty$. 
    Note that weak convergence of $\{x_n\}_n$ implies boundedness of the sequence $\{x_n\}_n$. 
    Moreover, by the embedding $H^1((0,T),\R^n) \hookrightarrow C([0,T],\R^n)$ we conclude that $x_n(0) \to x(0)$ and further we obtain $(x_{1,0})_n \to x_{1,0}$ by the convergence of $w_n$, which completes the proof.
\end{proof}
\begin{theorem}
    Let $\lambda >0$ or $\Uad$ bounded. Then, the 
    problem~\eqref{problem} admits a global minimum.
\end{theorem}
\begin{proof}
    With the continuous dependence on the date shown in Lemma~\ref{lemma1} and the weak continuity of the state operator discussed in Lemma~\ref{lemma2}, the statement follows as in \cite{guntherDatadrivenAdjointbasedCalibration2024}.
\end{proof}
We note that the global minimum does not have to be unique since our problem is generally non-convex. The non-convexity stems from the nonlinearity introduced by the reformulation using the inverse of $(J_{22} - R_{22})$.
In the following section, we will discuss how to find a global minimum using a gradient-based algorithm. 

\section{Numerical Optimization}\label{sec:optimization}
Here, we show how to determine an optimal parameter vector $w_{\rm opt}$ that minimizes the cost functional locally, by means of  
an iterative, gradient-based algorithm. 
The main challenge is to remain within the admissible set $\Uad$ at each step. Below, we discuss two approaches: a barrier term in the cost functional and a heuristic approach that exploits the structure of $J_{22}$ and $R_{22}$. 
In the following we assume that the dimension of the internal state $n$ and the number of differential variables $r$ is known. If this is not the case, it may be necessary to solve several optimization problems with different hyperparameters $n$ and $r$. 
We will examine this in the numerical examples.
Let us recall that the minimization problem reads
\begin{equation*} 
   \min \hat{\J}(w), \qquad w\in \Uad, 
\end{equation*}
with
\begin{equation*}
\Uad :=\Big\{ w = (J,R,B,x_{1,0}) \in W \colon \quad J^\top = -J,\quad R \succeq 0,\quad \sigma_{min}(J_{22}-R_{22}) \ge \epsilon >0 \Big\}.
\end{equation*}
Neglecting the regularity condition, the set $\Uad$ describes a type of matrix manifold.
Consequently, optimization techniques on matrix manifolds can be applied, as was done in \cite{absilOptimizationAlgorithmsMatrix2008}. Since we have to deal with the regularity constraint, which couples the conditions on $J$ and $R$, we choose to follow a parametrization approach \cite{cherifiNonlinearPortHamiltonianSystem2025} in the following.

That is, instead of restricting the optimization to skew-symmetric matrices $J$, we optimize over all matrices $K \in \R^{n\times n}$ and define $J = K-K^\top$. 
Similarly, the symmetric positive semi-definite matrix $R$ can be parameterized by $LL^\top$ (where $L \in \R^{n\times n}$), since $x^\top (LL^\top)x=\|L^\top x\|^2 \ge 0$. 
One disadvantage here is that the linear structure is lost and there are more degrees of freedom. 
To avoid increasing the degree of freedom, we could restrict $L$ to lower triangular matrices and $K$ to strictly lower triangular matrices.
However, for simplicity, we will use the parametrization introduced above and consider
\begin{equation*} 
   \min \J(\bar{w}), \qquad \bar{w}=(K,L,B,x_{1,0}) \in W := \R^{n\times n} \times \R^{n\times n} \times \R^{n\times m} \times \R^r, \end{equation*}
subject to
\begin{align*}
    E\dot{x} &= (K-K^\top - LL^\top) x + Bu, \qquad x_1(0)=x_{1,0}\in \R^r, \\
    y &= B^\top x.
\end{align*} 
In the numerical implementation, we concatenate the matrices $L,K$ and $B$ as well as the initial condition $x_{1,0}$ and work with $v\in\R^z$ where $z=2n+mn+r$ instead.

%

The main challenge is to deal with the constraint that $\sigma_{min}(J_{22}-R_{22})\ge\epsilon$ or $J_{22}-R_{22}$ is regular. Below we will discuss two possible approaches. The first one is based on a barrier term, which is added to the cost functional to prevent the iterates $v^k$  from approaching the singularities, see \cite{nocedalNumericalOptimization2006a}.
In our setting, a possible penalty term is given by
\begin{equation}\label{eq:mudet}
    - \mu \log\det(A^\top A), \qquad \text{with}\quad A=(J_{22}-R_{22}) \quad\text{and}\quad\mu >0.
\end{equation}
We can rewrite the term in \eqref{eq:mudet} as 
\begin{equation*}
    -\mu \log \det(A^\top A) = - \mu\log(\det (A)^2) = -2 \mu \log(|\det(A)|) = -2\mu \log(|\det (J_{22}-R_{22})|).
\end{equation*}
Specifically, as $J_{22} - R_{22}$ approaches a singularity, we have $|\det(J_{22} - R_{22})| \to 0^+$.
Consequently, the penalty term tends to infinity, which prevents the iterates from reaching singular matrices.
Furthermore, this method is consistent with our original condition, since 
\begin{equation*}
    \log\det(A^\top A) = \log\prod_{i=1}^k \sigma_i^2 
    = \sum_{i=1}^k \log \sigma_i^2.
\end{equation*}
However, this approach has two disadvantages. First, a sequence of $\{\mu_i\}$ converging to zero must be chosen to solve the original problem
\begin{equation*}
    \min \J_1(y) + \J_2(w) + \bigl(-\mu_i \log(|\det (J_{22}-R_{22})|)\bigr) \qquad \text{with}\quad \mu_i \to 0  
    \quad\text{for}\quad i=1,2,\dots .
\end{equation*}
On the other hand, the gradient of the term $-\log |\det (A)|$ is given by $A^{-\top}$. Consequently, the full inverse $(J_{22}-R_{22})^{-1}$ must be computed at every evaluation of $\hat{\J}$, which can be very expensive.
Specifically, solving the matrix equation $AX=I_{n-r}$ is required to obtain $A^{-1} \coloneq X\in \R^{(n-r)\times (n-r)}$. 

The second idea for avoiding singularities is a heuristic approach. 
Since the matrix $J_{22}$ is skew-symmetric and $R_{22}$ is positive semi-definite, we have $x^\top J_{22}x - x^\top R_{22}x = -x^\top R_{22}x \le 0$ for all $x\in \R^{n-r}$. Consequently, $J_{22}-R_{22}$ can only be singular if there exists a nonzero vector $x$ such that $J_{22}\,x=0$ and $R_{22}\,x=0$ and therefore $\operatorname{ker} (J_{22}) \cap \operatorname{ker} (R_{22})$ is nontrivial. 
Since the solution to our problem is computed iteratively using a gradient-based method, it is difficult to keep track of this condition. However, suppose that $v^k$ is an admissible iterate, meaning that $J_{22}-R_{22}$ is regular. 
We then compute a new iterate $v^{k+1}$ using a search direction $s^k$ and a step size $\sigma_k$, which indicates how far we move in that direction, and set $v^{k+1}\coloneq v^k+\sigma_k s^k$. 
Note that the cost functional can only be evaluated if the state equation is solvable. 
One way to avoid singularities is therefore to return an error message when the state equation cannot be solved and reject the corresponding step size. 

\begin{remark}
    We report that $v^k \in \operatorname{ker} (J_{22}) \cap \operatorname{ker} (R_{22})$ rarely occurred in the numerical examples.
\end{remark}

\subsection{The Gradient Descent Method}
Since we reformulated our identification problem as an unconstrained optimization problem, we aim to apply standard techniques from unconstrained nonlinear optimization (see e.g.\ \cite{hinzeOptimizationPDEConstraints2009,nocedalNumericalOptimization2006a}) to the reduced cost functional $\hat \J(w)$. Since its definition involves the state solution operator, we determine $\nabla \hat \J$ with the help of an adjoint-based approach. In \cite{guntherStructurepreservingIdentificationPortHamiltonian2023}, this approach was used for to solve the calibration task for pH-ODEs.
We follow the lines of \cite{guntherDatadrivenAdjointbasedCalibration2024,hinzeOptimizationPDEConstraints2009} and define the Lagrange function $L\colon H^1((0,T),\R^n))\times \Uad \times Z^*\to\R$ given by
\begin{equation*}
   L(x,w,\bm{p}) = \J(x,w) - \langle \bm{p},e(x,w) \rangle_{Z^*,Z},
\end{equation*}
where $\bm{p}\in Z^*$ denotes the Lagrange multiplier or adjoint variable, and $e(x,w)$ is the state operator \eqref{eq:e}. 
Note that  $Z=L^2((0,T),\R^n)\times \R^r$ and the dual space is given by $Z^*=L^2((0,T),\R^n)\times \R^r$. 
Thus, the dual pairing can be written as follows:
\begin{equation*}
    \langle \bm{p}, e(x,w) \rangle_{Z^*,Z}
    =  \langle p, \tfrac{\dd}{\dd t}x - (J-R)x - Bu  \rangle_{L^2((0,T),\R^n)} + \langle p_{1,0}, x_1(0)-x_{1,0} \rangle_{\R^r}.
\end{equation*}
For simplicity, we use the notation $\bm{p}=(p,p_{1,0}) \in Z^*$.
The first-order optimality system is derived by solving $dL=0$. For $h_x \in H^1((0,T),\R^n)$, we compute the directional derivative
\begin{align*} 
d_x &L(x,w,\bm{p})[h_x] 
= d_x\J(x,w)[h_x] - d_x \langle e(x,w)[h_x],\bm{p} \rangle \\
&= \frac{1}{2}\int_0^T d_x |B^\top x-\yd|^2 [h_x] \,\dd t\\
& \qquad -  \int_0^T d_x\biggl( \Bigl( E\frac{\dd}{\dd t}x - (J-R)x - Bu \Bigr) \cdot p \biggr) [h_x] \, \dd t - d_x \bigl((x_1(0)-x_{1,0}) \cdot p_{1,0}\bigr)[h_x] \\
%
&= \int_0^T B(y-\yd) \cdot h_x + (J-R)h_x \cdot p \, \dd t - \int_0^T \Bigl( E\,\frac{\dd}{\dd t}h_x\Bigr) \cdot p \, \dd t  - h_{1,x}(0)\cdot p_{1,0}
\end{align*}
Next, we use integration by parts for the second integral and obtain
\begin{align*}
    &\int_0^T B(y -\yd) \cdot h_x + (J-R)h_x \cdot p \, \dd t - \big[ E h_x\cdot p \big]_0^\top + \int_0^T E h_x \cdot \frac{\dd}{\dd t}p \, \dd t- h_{1,x}(0) \cdot p_{1,0} \\
    &=\int_0^T B(y-\yd) \cdot h_x + E\frac{\dd}{\dd t}p \cdot h_x + (J^\top - R^\top)p \cdot h_x \dd t - p_1(T)\cdot h_x(T).
\end{align*}
In the second equation, we set $p_{1,0} \coloneq p_1(0)$ and used $Eh_x(0)\cdot p(0) = h_x(0)^\top \diag(I_r,0)p(0) \\ = h_{1,x}(0)^\top p_{1,0}$.
As $d_x L(x,w,\bm{p})[h_x] = 0$ must hold for arbitrary $h_x$ we identify the \textit{adjoint equation} as
\begin{equation}\label{eq:adjoint}
    -E\frac{\dd}{\dd t} p = (J^\top - R^\top) p + B(y- \yd), \qquad p_1(T)=0.
\end{equation}
Like the state equation, this adjoint equation is an pH-DAE of index-1, since $J^\top$ is a skew-symmetric matrix and $R^\top = R \succeq0$. 
The index-1 condition follows directly from $\det((J_{22}-R_{22})^\top )= \det(J_{22}-R_{22})\not = 0$.
The difference is that the equation is solved backward in time with end condition $p_1(T)=0$, and $p_{2,T}$ is calculated similarly to $x_{2,0}$. 

The derivative $d_{\bm{p}} L(x,w,{\bm{p}})[h_{\bm{p}}] = 0$ yields the state equation \eqref{eq:state}. 
The optimality condition is derived as
\begin{equation}
    \begin{split}  
        d_w L(x,w,{\bm{p}})[h_w] &= d_w\J(x,w)[h_w] - d_w \langle e(x,w),{\bm{p}} \rangle[h_w] \\
        &= d_w\J_1(x,w)[h_w] + d_w\J_2(w)[h_w] - d_w \langle e(x,w),{\bm{p}} [h_w]\rangle.
    \end{split}
\end{equation}
For simplicity, we split the derivation into the different matrices and the initial condition. Moreover, we omit the term $\J_2$ during the calculations and will later add the term $\nabla \J_2$ to obtain the full gradient.
For the derivative w.r.t.\ to $x_{1,0}$, we obtain
\begin{equation}
     d_{x_{1,0}} \langle e(x,w),{\bm{p}} \rangle [h_{x_{1,0}}]=  \langle p_{1,0},h_{x_{1,0}}\rangle_{\R^r} = (p_{1,0},h_{x_{1,0}})_{\R^r}.
\end{equation}
Since $h_{x_{1,0}}$ was chosen arbitrarily, the Riesz representation theorem 
allows us to identify $\nabla_{x_{1,0}} \hat{\J}(w)= p_1(0) +\nabla_{x_{1,0}} \J_2(w)$. The gradients w.r.t.\ the matrices are derived analogously.
Thus, for the derivative with respect to matrix $J$, we obtain
with the Frobenius inner product
\begin{equation*}
\begin{split}
    d_J \langle e(x,w),{\bm{p}} \rangle[h_J] 
    &= -\int_0^T d_J \biggl( \Bigl (E \frac{\dd}{\dd t} - (J-R)x - Bu \Bigr)\cdot p \biggr) [h_J]\, \dd t \\
    &=  \int_0^T h_J x \cdot p \,\dd t = \int_0^T p^\top h_J x  \,\dd t 
    = \int_0^T \tr (p^\top h_J x)  \,\dd t 
    = \int_0^T \tr (x p^\top h_J) \, \dd t \\
    &= \int_0^T  \bigl \langle p x^\top, h_J\bigr \rangle_F \;\dd t = \Bigl\langle \int_0^T  p x^\top \, \dd t, h_J\Bigr\rangle_F .
\end{split}
\end{equation*}
From the last equation, we conclude that the gradient with respect to $J$ is equal to $\int_0^T  p x^\top \,\dd t$. 
The same calculation applies to matrix $R$, but with the opposite sign. The last directionally derivative with respect to matrix $B$ is given by
\begin{equation*}
    \begin{split}
        d_B& \J_1(x,w)[h_B] - d_B \langle e(x,w),{\bm{p}} \rangle [h_B] \\
        &= \frac{1}{2}\int_0^T d_B |B^\top x-\yd|^2 [h_B] \;\dd t - \int_0^T d_B \biggl( \Bigl(E \frac{\dd}{\dd t} - (J-R)x - Bu \Bigr)\cdot p \biggr) [h_B] \;\dd t\\ 
        &= \int_0^T x^\top h_B (B^\top x-\yd) + p^\top  h_B u \; \dd t 
        = \int_0^T \tr \bigl(x^\top h_B (y-\yd)\bigr) + \tr (p^\top h_B u) \; \dd t \\
        %
        %
        %
        &= \int_0^T \bigl\langle x(y-\yd)^\top, h_B  \bigr\rangle_F + \bigl\langle pu^\top, h_B  \bigr\rangle_F \; \dd t 
        =\Bigl\langle \int_0^T  x(y-\yd)^\top + pu^\top \dd t, h_B  \Bigr\rangle_F .
    \end{split}
\end{equation*}
The Riesz representation theorem once again allows us to identify the gradient
    \begin{align}
            0&=\nabla_J \J_2(\bar{w}) + \int_0^T \bar{p}(t) \otimes \bar{x}(t) \dd t, \\
            0&=\nabla_R \J_2(\bar{w}) - \int_0^T \bar{p}(t) \otimes \bar{x}(t) \dd t, \\
            0&=\nabla_B \J_2(\bar{w}) + \int_0^T \bigl( \bar{p}(t) \otimes u(t) + (\bar{x}(t)) \otimes (\bar{y}(t)-\yd(t))  \bigr) \dd t, \\
            0&=\nabla_{x_{1,0}} \J_2(w) + \bar{p}_1(0),
        \end{align}
        where $\bar{p}$ satisfied the adjoint equation
\begin{equation}
    -E\frac{\dd}{\dd t} \bar{p} = (\bar{J}^\top - \bar{R}^\top) \bar{p} + \bar{B}(\bar{y}- \yd), \qquad \bar{p}_1(T)=0,
\end{equation}
and $\bar{x}$ the state equation with output $\bar{y}$ given by
\begin{subequations}
    \begin{align*}
        E\dfrac{\dd}{\dd t} \bar{x} &= (\bar{J} - \bar{R})\bar{x} + \bar{B}u, \qquad \bar{x}_1(0) = \bar{x}_{1,0}, \\
        \bar{y} &= \bar{B}^\top \bar{x},
    \end{align*}
\end{subequations}
where $a \otimes b$ denotes the dyadic product $ab^\top$ for $a,b \in \R^n$.

Now, we apply the transformation we previously introduced and obtain for $J=K-K^\top$ and $R=L L^\top$ that it holds
\begin{equation}\label{eq:chain_rule}
    \begin{aligned}
       \nabla_K \hat{\J}(w) &= \nabla_J \hat{\J}(w) - \nabla_J \hat{\J}(w) ^\top, \\
        \langle \hat{\J}'(w),h_L\rangle_F 
        &= \tr(\hat{\J}_R(w)^\top h_L L^\top ) + \tr (\hat{\J}_R(w)^\top L h_L^\top)
        =  \langle \nabla_R \hat{\J}(w) L + \nabla_R \hat{\J}(w)^\top L, h_L \rangle_F  \\
       &\Rightarrow \nabla_L \hat{\J}(w) = \bigl(\nabla_R \hat{\J}(w) + \nabla_R \hat{\J}(w)^\top \bigr) L.
    \end{aligned}
\end{equation}
Due to the transformation, the problem is unconstrained and we can summarize the derivation with the help of the following first-order optimality condition.

\begin{theorem}
    A stationary point $\bar{v}=(\bar{K},\bar{L},\bar{B},\bar{x}_{1,0})$ of the transformed reduced problem satisfies the optimality condition
    \begin{subequations}\label{eq:nablaJ}
        \begin{align}
            0&=\nabla_K \hat{\J}(\bar v) = \nabla_K \J_2(\bar{v}) + \int_0^T \bar{p}(t) \otimes \bar{x}(t) \dd t - \int_0^T \bar{x}(t) \otimes \bar{p}(t) \dd t, \\
            0&= \nabla_L \hat{\J}(\hat v)  =\nabla_L \J_2(\bar{v}) - \left(\int_0^T \bar{p}(t) \otimes \bar{x}(t) \dd t + \int_0^T \bar{x}(t) \otimes \bar{p}(t) \dd t \right) \bar{L}, \\
            0&=\nabla_B \J_2(\bar{w}) + \int_0^T \bigl( \bar{p}(t) \otimes u(t) + (\bar{x}(t)) \otimes (\bar{y}(t)-\yd(t))  \bigr) \dd t, \\
            0&=\nabla_{x_{1,0}} \J_2(w) + \bar{p}_1(0),
        \end{align}
    \end{subequations}
where $\bar{p}$ satisfied the adjoint equation
\begin{equation}
    -E\frac{\dd}{\dd t} \bar{p} = (\bar{K}^\top -\bar{K} - \bar{L}^\top\bar{L}) \bar{p} + \bar{B}(\bar{y}- \yd), \qquad \bar{p}_1(T)=0,
\end{equation}
and $\bar{x}$ the state equation with output $\bar{y}$ given by
\begin{subequations}
    \begin{align*}
        E\dfrac{\dd}{\dd t} \bar{x} &= (\bar{K} - \bar{K}^\top - \bar{L}\bar{L}^\top)\bar{x} + \bar{B}u, \qquad \bar{x}_1(0) = \bar{x}_{1,0}, \\
        \bar{y} &= \bar{B}^\top \bar{x},
    \end{align*}
\end{subequations}
where $a \otimes b$ denotes the dyadic product $ab^\top$ for $a,b \in \R^n$.
\end{theorem}

Note that it is difficult to solve the first-order optimality system altogether due to its forward-backward structure. 
Instead, we proceed iteratively and first solve the state equation in the forward direction, followed by the adjoint equation in the backward direction. 
The gradient \eqref{eq:nablaJ}, which is computed based on the state and adjoint variables, then allows us to define a suitable descent direction.
Finally, we formulate all steps in Algorithm \ref{alg_with_gradient}.
\vspace{1em}

\begin{algorithm}[H]
\caption{Algorithm for the calibration process}\label{alg_with_gradient}
\KwData{initial guess $v$, algorithmic parameters, stopping criterion}
\KwResult{optimized matrices and initial data $\bar{w}=(\bar{J},\bar{R},\bar{B},\bar{x}_{1,0})$}
\SetCustomAlgoRuledWidth{1cm}
\SetAlgoLined
1) set $J=K-K^\top$ and $R=LL^\top$\;
2) solve state equation \eqref{eq:state} $\to x$\;
3) solve adjoint equation \eqref{eq:adjoint} $\to p$\;
4) evaluate the gradient \eqref{eq:nablaJ} and \eqref{eq:chain_rule}$\to \nabla\J(w)$\;
5) transform $\nabla\J(w)$ into a vector\;
5) compute the direction of descent via conjugate gradient method \cite{nocedalNumericalOptimization2006a} $\to s$\; 
6) find an admissible step size via Armijo rule \cite{hinzeOptimizationPDEConstraints2009} $\to \sigma$\;
7) update $v \mapsto v + \sigma s$\;
8) \textbf{if} stopping criterion is not fulfilled $\to$ go to 1)\; 
$\quad$\textbf{else} return optimal control $\bar{w}$\;
\end{algorithm}
Possible stopping criteria are $\|\nabla f(x)\| <\varepsilon$ for $0<\varepsilon \ll 1$ or when the maximal number of iterations is reached. 
Now, we will formally prove that the calculated gradient corresponds to the gradient, obtained by directly differentiating the reduced cost functional $\hat{\J}$. 
Note that the state equation \eqref{eq:state} is treated implicitly, and that $x$ and $y$ can be expressed in terms of $S(w)$ and $B^\top x$, respectively. To identify the gradient of the reduced cost functional, we first differentiate the equation $e(x,w)=e(S(w),w)=0$ with respect to $w$ in the direction $h \in W$ to obtain using the implicit function theorem
\begin{align*}
      &d_x e(S(w),w) S'(w)[h]+ d_w e(S(w),w)[h] = 0 \,
      \Rightarrow  \, S'(w)[h] = -d_x e(S(w),w)^{-1} d_w e(S(w),w))[h].
\end{align*}
This allows us to find the directional derivative of $\hat{\J}(w)=\J(x,w)$:
\begin{align*}
d\hat{\J}(w)[h] 
&= \langle d_y \J(y,w), B^\top S'(w)[h] \rangle_{H^{-1},H^1} + \langle d_w \J(y,w), h\rangle_{W^*,W} \\
&=\langle-  d_w e(S(w),w)^*d_x e(S(w),w)^{-*} [B d_y \J(y,w)], h  \rangle_{W^*,W} + \langle d_w \J(y,w), h\rangle_{W^*,W} \\
&= \langle d_w \J(y,w) -d_w e(S(w),w)^* d_x e(S(w),w)^{-*} [B\,d_y \J(y,w)], h \rangle_{W^*,W}.
\end{align*}
Since $W$ is a Hilbert space and $h$ was chosen arbitrarily, the Riesz 
theorem yields 
\begin{equation*}
   (\nabla \hat{\J}(w),h)_W = \langle d_w \J(y,w) -d_w e(S(w),w)^* d_x e(S(w),w)^{-*} [B\,d_y \J(y,w)], h \rangle_{W^*,W}.
\end{equation*}
The gradient is therefore given by
\begin{subequations}
    \begin{align}
        \nabla \J(w) &= d_w \J(y,w) -d_w e(S(w),w)^* \underbrace{d_x e(S(w),w)^{-*} [B\,d_y \J(y,w)]}_{=:p} \\
        &= d_w \J(y,w) -d_w e(S(w),w)^*p. \label{eq:grad}
    \end{align}
\end{subequations}
In this context,
    $d_x e(S(w),w)^* p = B\,d_y \J(y,w)$
describes the adjoint equation~\eqref{eq:adjoint} and the equation~\eqref{eq:grad} is exactly the gradient calculated above, i.e.\ the left-hand side of system~\eqref{eq:nablaJ}.

\section{Numerical Results}\label{sec:numerical}

For the numerical results we consider the interval $[0,T]$ with $T=1$ and 1001 discretization points, resulting in a step size of $h=0.0001$.
\begin{figure}[htbp]
    \centering
    \includegraphics[scale=0.45]{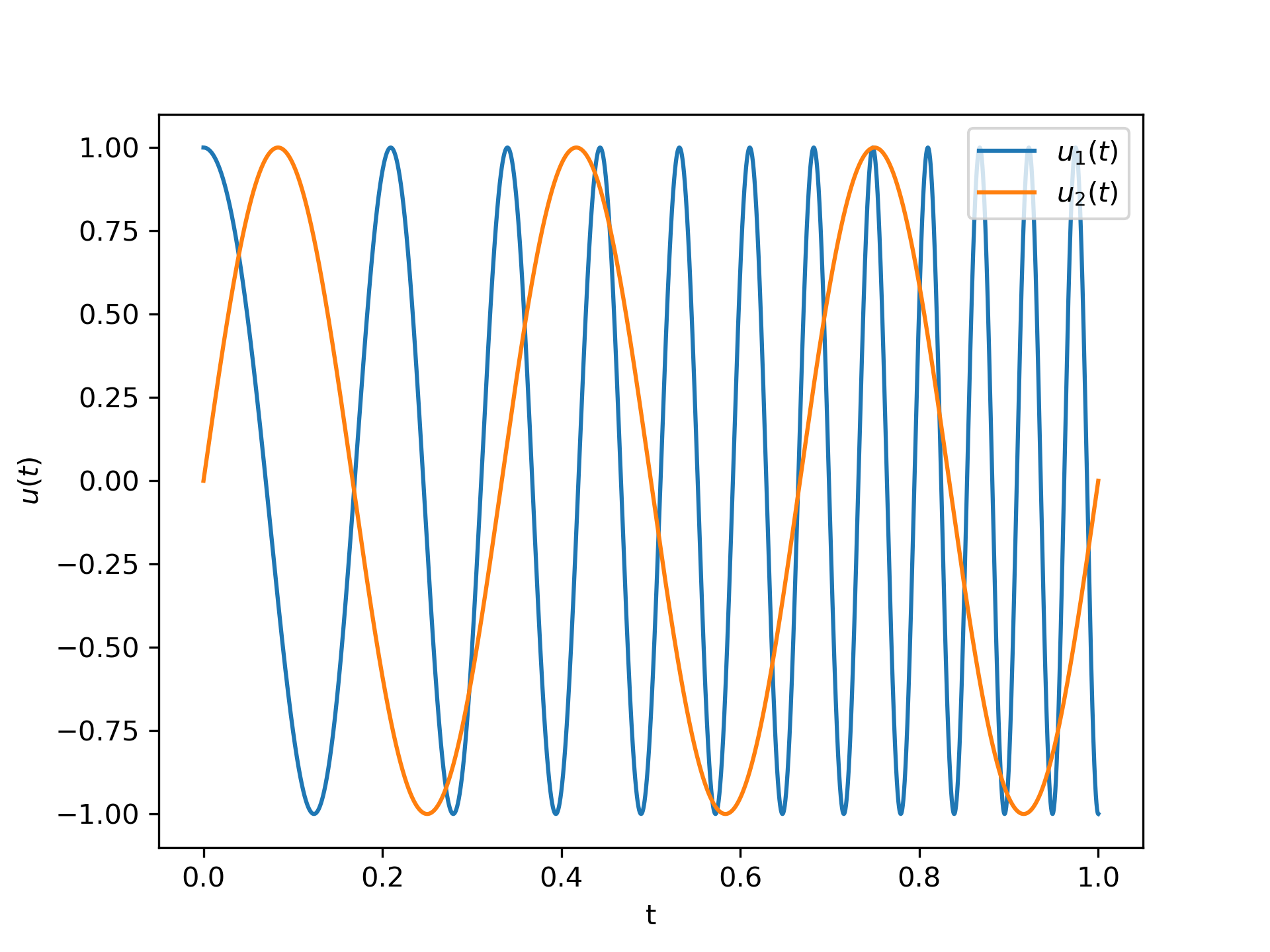}
    \includegraphics[scale=0.45]{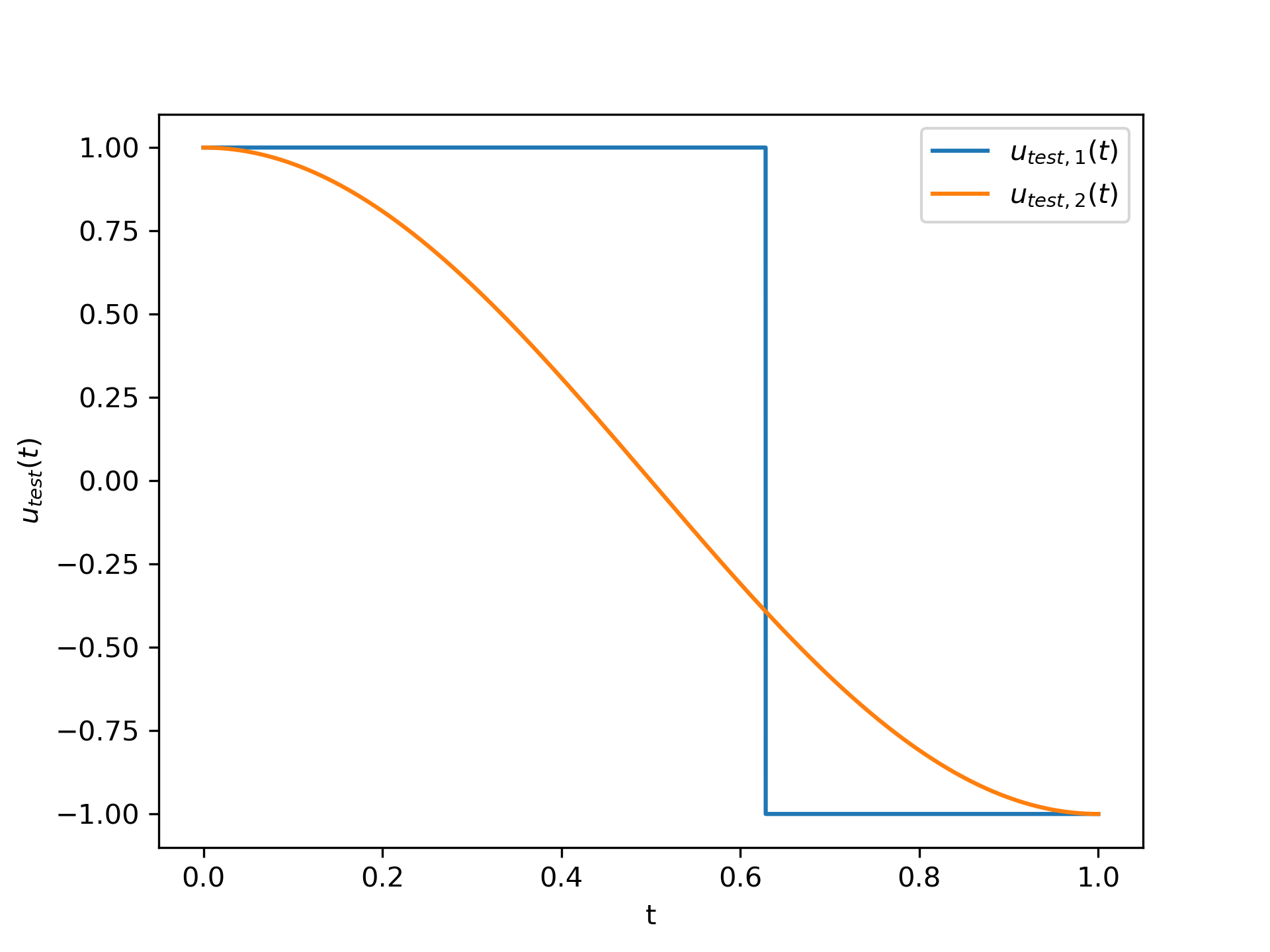}
    \caption{Signal $u$ for the calibration (left) and signal $u_\text{test}$ for cross-validation (right)}
    \label{fig:u_input}
\end{figure}
We set $m=2$ for all numerical test and used the same inputs as in \cite{guntherDatadrivenAdjointbasedCalibration2024} (see Figure~\ref{fig:u_input}) 
\begin{align*}
    u\colon [0,T]\to \R^2,\qquad u(t)=\begin{pmatrix}
        \operatorname{chirp}_{3,20}(t) \\ 
        \sin (2\pi f_0 t)
    \end{pmatrix},
\end{align*}
where $\operatorname{chirp}$ denotes the chirp-signal from \texttt{scipy.signal} toolbox with frequency $3$ at time $t=0$ and frequency $20$ at time $t=1$ and $f_0 = 3$.
To check the robustness of the identified system, we perform a cross-validation using the test signal (see Figure~\ref{fig:u_input})
\begin{equation*}
    u_\text{test} \colon [0,T] \to\R^2, \qquad u_\text{test}(t) = \begin{pmatrix} \operatorname{square}_{0.1}(t) \\ \cos(\pi t) \end{pmatrix},
\end{equation*}
where $\operatorname{square}_{0.1}$ is the square signal of the \texttt{scipy.signal} toolbox with duty parameter $0.1$.

\subsection{Synthetic data}
Below, we present the results for a synthetic dataset with $n=5$ and $m=2$. 
To this end, we select $w_{\rm data}$ and use it to compute $y_{\rm data}$. 
We initialize $B_{\rm data}$ with independent uniformly distributed entries in $[-0.2,0.2]$ and add $\diag(4,2)$ to the first $m\times m$ block.
For the matrix $R_{\rm data}$, we draw an $n\times n$ matrix $L$ with independent uniformly distributed entries in $[-1, 1]$.
Then, we set the strictly upper triangular part of $L$ as well as the elements $l_{11} =l_{22}$ to zero. Finally, we define $R_{\rm data}=L^\top L$, which is positive semidefinite. 
To generate a feasible initial $J$, we draw an $n\times n$ matrix with independent uniformly distributed entries in $[-1, 1]$. We take the upper part of this matrix, starting with the first off-diagonal, and fill the lower part of the matrix with the negative of the transpose of the upper part. 
Since we restrict ourselves to index-1 DAEs, we must check if $J$ and $R$ are admissible. 
The initial value of the differential component $x_{1,0,\rm data}$ is a random vector with entries in $[-0.1, 1]$.

First, the number of differential variables $r$ is $4$, while in the second case, it is $2$. 
This allows us to investigate how the number of differential variables affects the identification process. 
For all tests, we assume to 
have no information about the system, so, we set $\lambda = 0$. 


\begin{example}[System with four Differential Variables]
In this test, we generated an output $\yd=y_{\rm ref}$ with $r_{\rm ref}=4$.
To investigate the influence of the choice of $r$, we solved a minimization problem for each $r \in \{1,2,3,4,5\}$.
The same random initial values for the matrices and the first r components of the random vector $x_{1,0}$ were used for each minimization problem.
Figure~\ref{fig:ex1_cost} shows the evolution of the cost for all five test cases. 
As can be seen, the cost functional for the system with one differential variable is the highest; in the other cases, the costs have decreased significantly. 
Table~\ref{tab:ex1_cost} lists the costs as well as the maximum norm error $\|y - y_{\rm ref}\|_\infty$ for the first and last iterations. 
We observe that the costs and the maximum norm error are lowest for $r=3$. 
For all test cases except $r=1$, however, the reduction in the cost functional is above 99\,\%. 
Figures~\ref{fig:ex1_start} and \ref{fig:ex1_output} show the two system outputs before and after the calibration together with the reference output (dashed yellow line) for $r=1,\dots,5$. 
As the error values already indicated, the fitted outputs match the reference output very well, except for the system with only one differential variable.

\begin{figure}[ht!]
    \centering
    \includegraphics[width=0.8\textwidth]{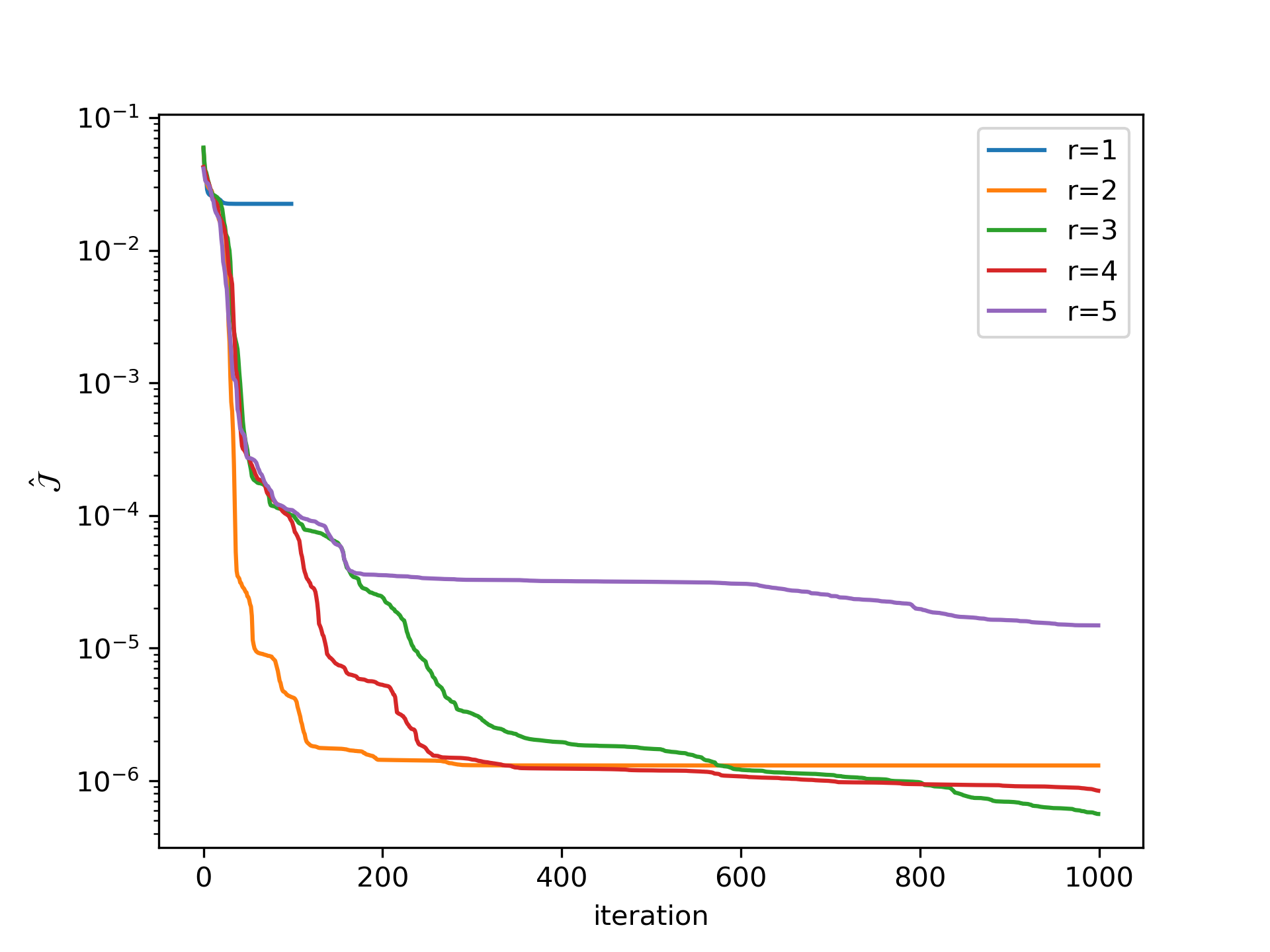}
    \caption{Evolution of the cost for various $r$ over the iterations of the calibration in log plot for $r_{\rm ref}=4$.}
    \label{fig:ex1_cost}
\end{figure}

\begin{table}
    \centering
    \begin{tabular}{c|c|c|c|c|c}
        number of differential variables $r$ & 1 & 2 & 3 & 4 & 5\\ \hline
        cost in first iteration &\num{5.75e-02} &\num{5.83e-02} &\num{5.95e-02} &\num{4.25e-02} &\num{4.14e-02} \\
        cost in last iteration &\num{2.24e-02} &\num{1.30e-06} &\num{5.62e-07} &\num{8.42e-07} &\num{1.49e-05}  \\
        $\|y - y_{\rm ref}\|_\infty$ in first iteration &\num{7.48e-01} &\num{7.70e-01} &\num{7.12e-01} &\num{7.86e-01} &\num{7.81e-01}  \\
        $\|y - y_{\rm ref}\|_\infty$ in last iteration &\num{4.88e-01} &\num{2.95e-03} &\num{2.41e-03} &\num{3.45e-03} &\num{9.84e-03} \\
    \end{tabular}
    \caption{Initial and final cost values, the $L^2$- and the maximum norm between the fitted output and reference output for various numbers of differential variables $r$ for $r_{\rm ref}=4$.}
    \label{tab:ex1_cost}
\end{table}

\begin{figure}[ht!]
    \centering
    \includegraphics[width=0.45\textwidth]{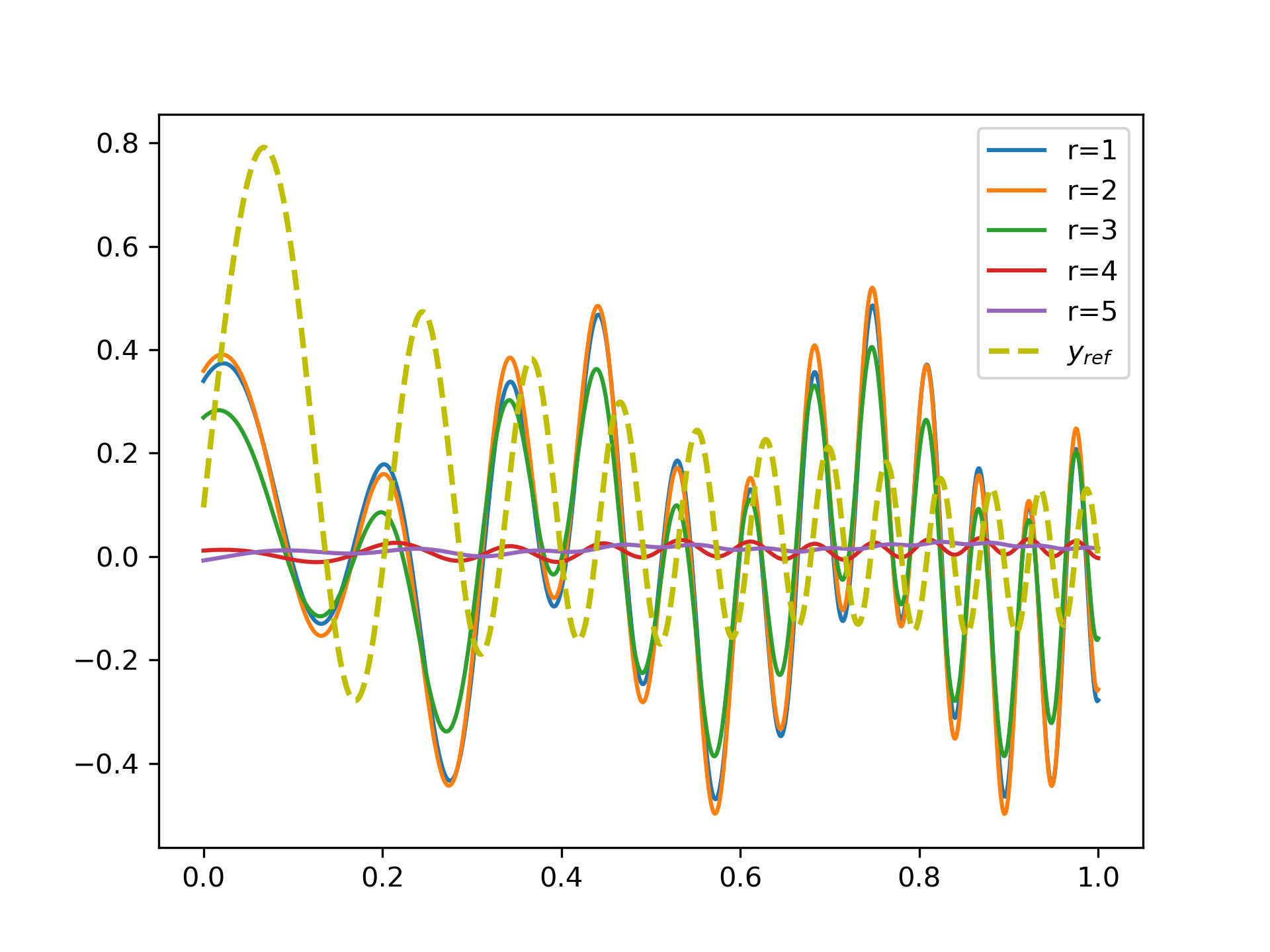}
    \includegraphics[width=0.45\textwidth]{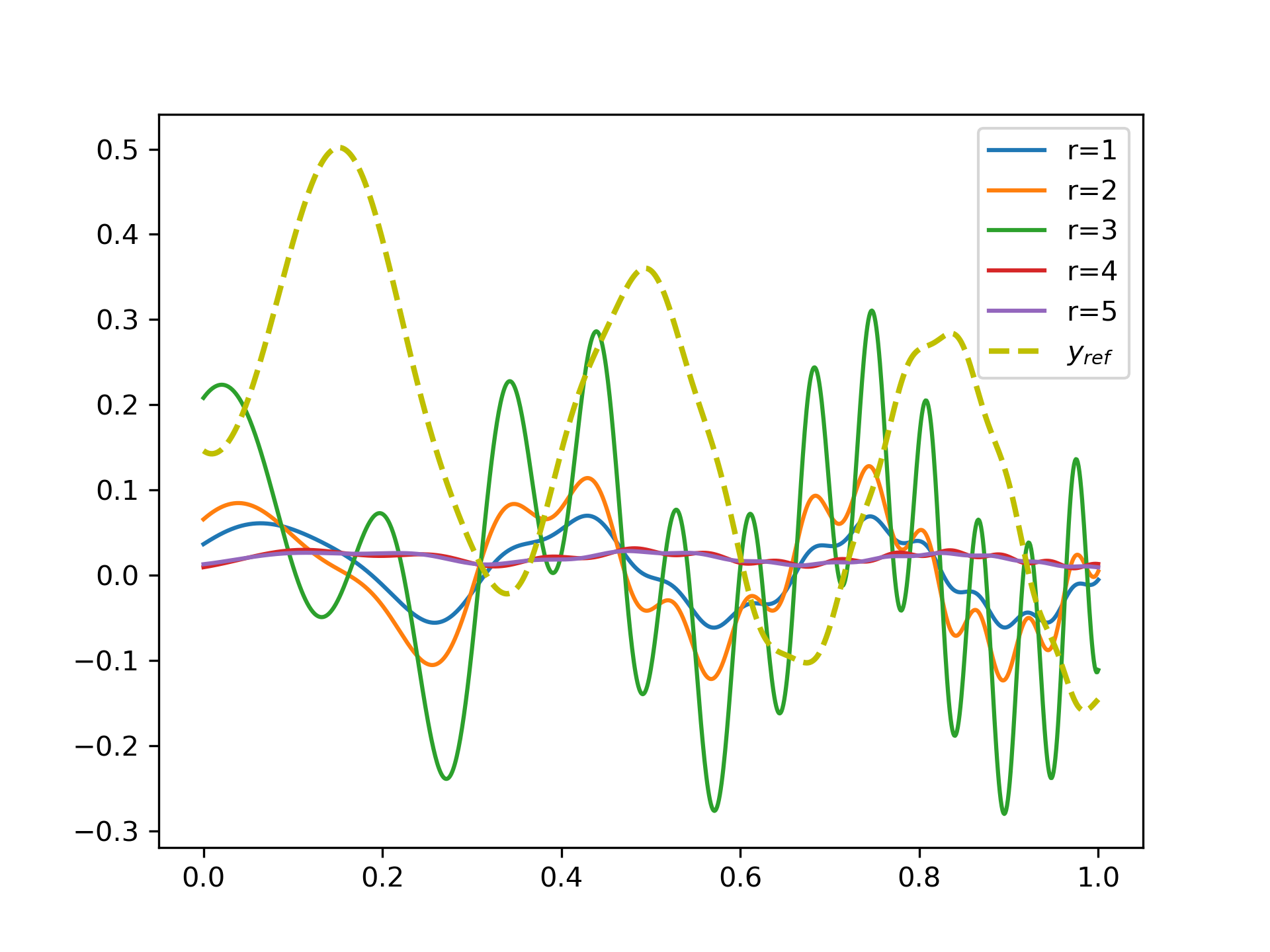}
    \caption{Left: First component of the output before the calibration for various $r$ and for the reference system $y_{\rm ref}$;  Right: Second component of the output before the calibration for various $r$ and for the reference system $y_{\rm ref}$ for $r_{\rm ref}=4$.}
    \label{fig:ex1_start}
\end{figure}

\begin{figure}[ht!]
    \centering
    \includegraphics[width=0.45\textwidth]{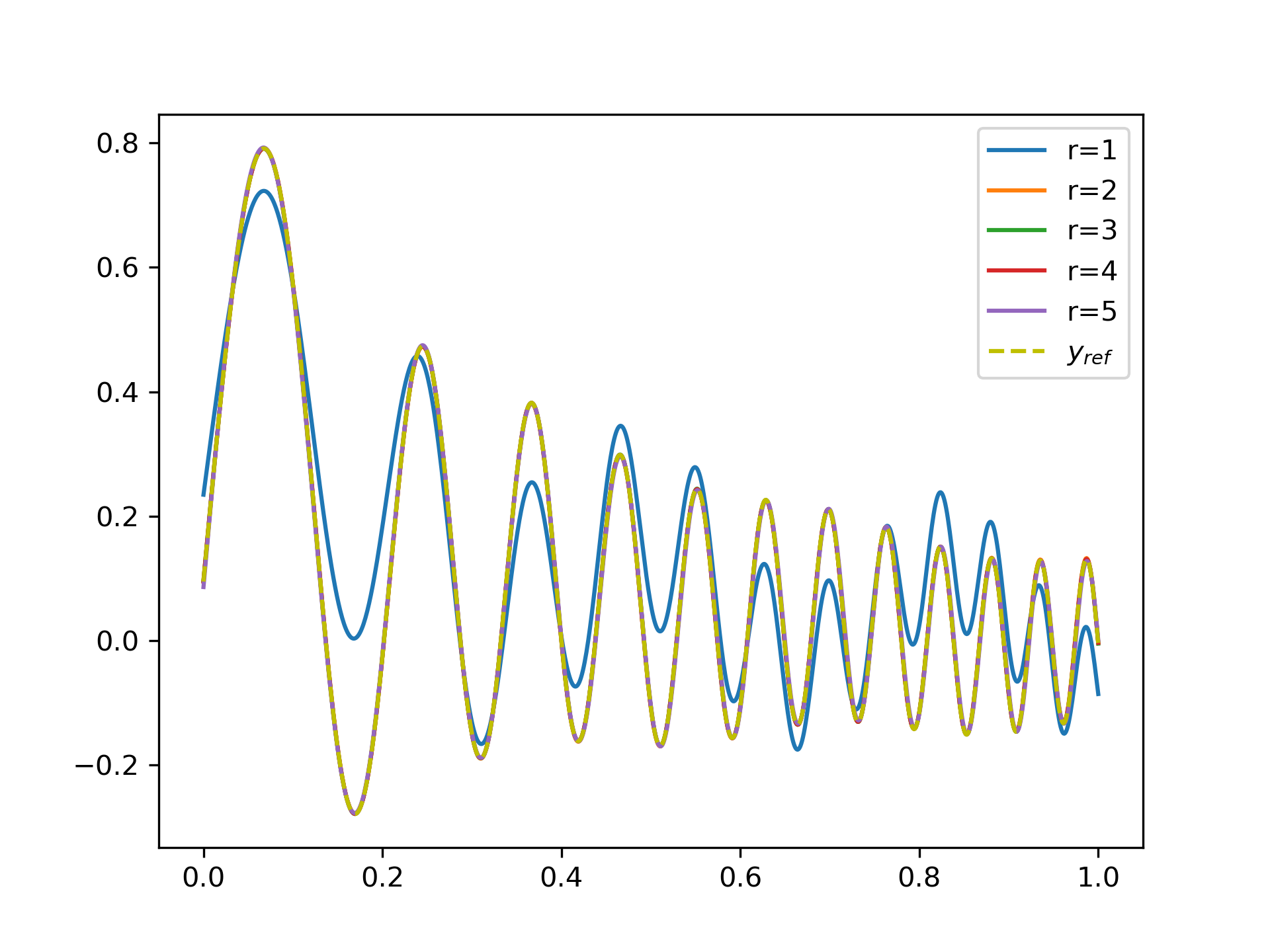}
    \includegraphics[width=0.45\textwidth]{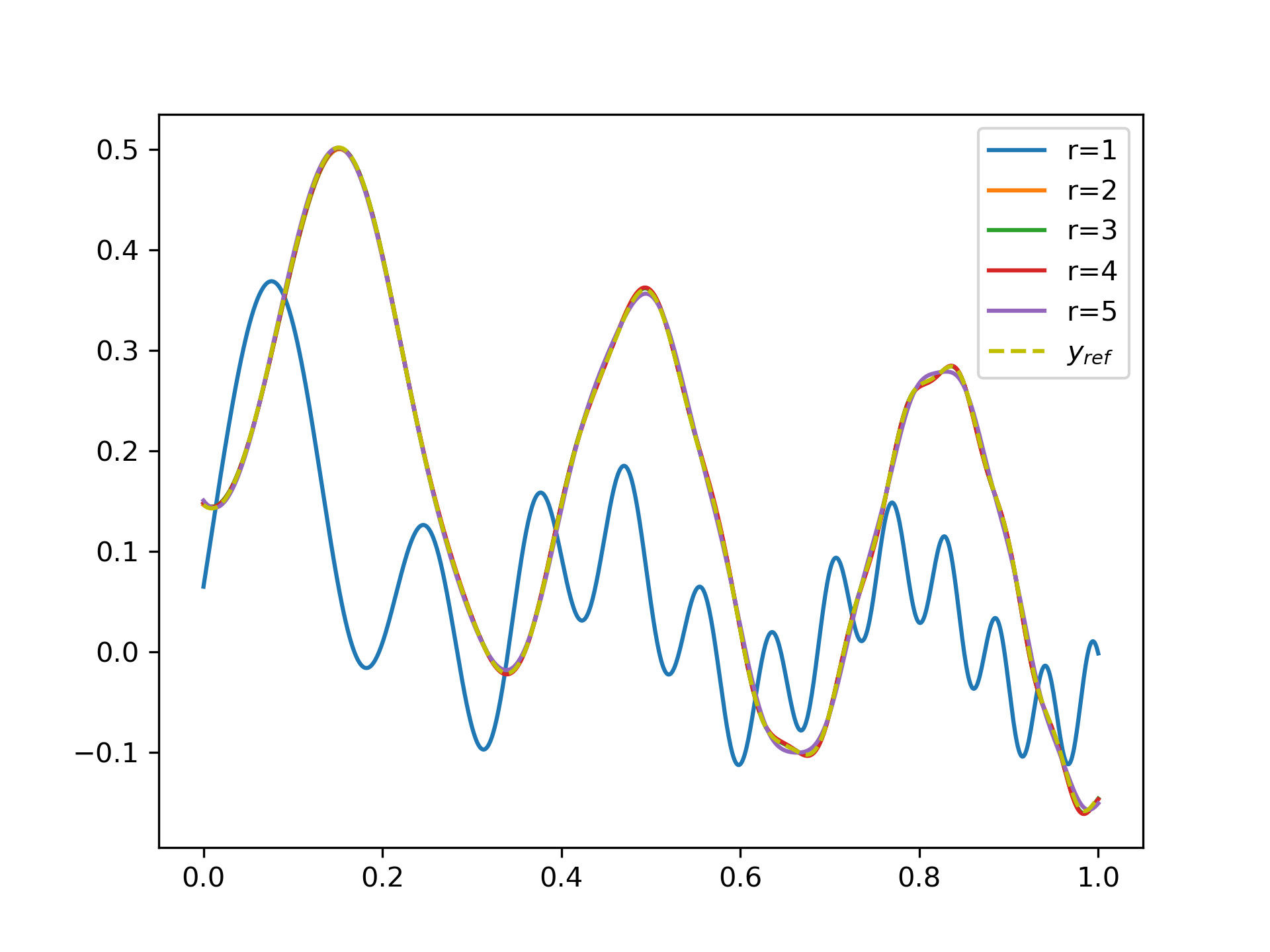}
    \caption{Left: First component of the output after the calibration for various $r$ and for the reference system $y_{\rm ref}$;
    Right: Second component of the output after the calibration for various $r$ and for the reference system $y_{\rm ref}$ for $r_{\rm ref}=4$.}
    \label{fig:ex1_output}
\end{figure}

Note that a single signal is insufficient for determining how well the fitted system represents the true system.
Therefore, we conclude this example with a cross-validation test. 
We apply the input signal $u_\text{test}$ to the fitted and reference systems, and we compare the corresponding outputs.
Table~\ref{tab:ex1_cross} shows the $L^2$-difference between $y$ and $y_\text{test}$ as well as the maximum norm. 
As we can see, the ODE system ($r=5$) has the largest error in the cross-validation test, while the error for $r=4$ is the smallest and closest to the actual value.
Figure~\ref{fig:ex1_cross} shows the results of this cross-validation test graphically and the outputs for the fitted systems $r=2,3$ also match well.

\begin{table}[H]
    \centering
    \begin{tabular}{c|c|c|c|c|c}
        number of differential variables $r$ & 1 & 2 & 3 & 4 & 5\\ \hline
        $\|y - y_\text{test}\|_{L^2}$ &\num{3.96e+00} &\num{3.29e-01} &\num{4.42e-01} &\num{2.62e-01} &\num{5.11e+00}  \\
        $\|y - y_\text{test}\|_\infty$  &\num{5.73e+00} &\num{6.25e-01} &\num{6.26e-01} &\num{4.60e-01} &\num{8.67e+00}  \\
    \end{tabular}
    \caption{The $L^2$-difference and the maximum norm between the fitted output and reference output for various $r$ in the cross-validation test and $r_{\rm ref}=4$.}
    \label{tab:ex1_cross}
\end{table}

\begin{figure}[H]
    \centering
    \includegraphics[width=0.45\textwidth]{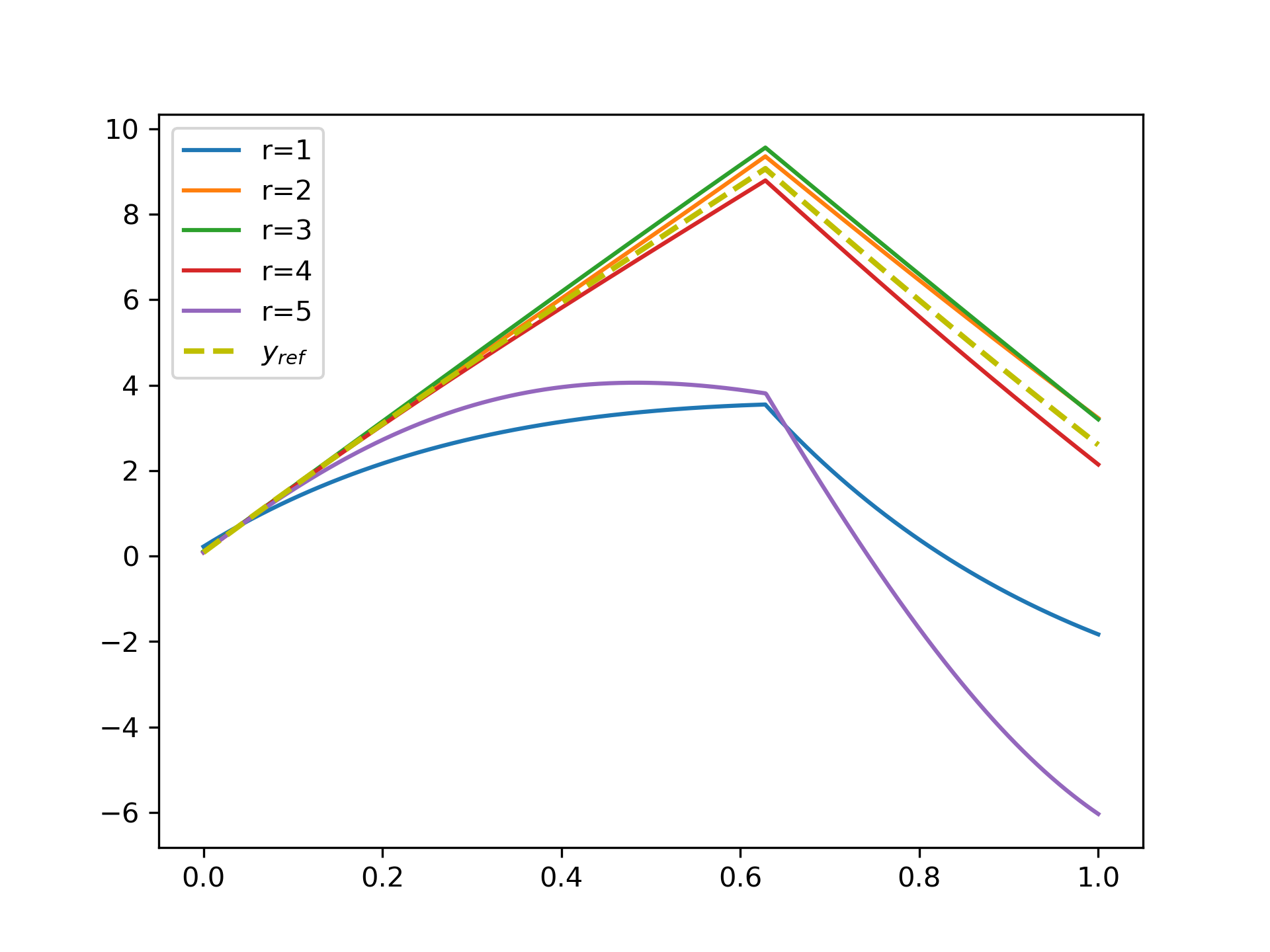}
    \includegraphics[width=0.45\textwidth]{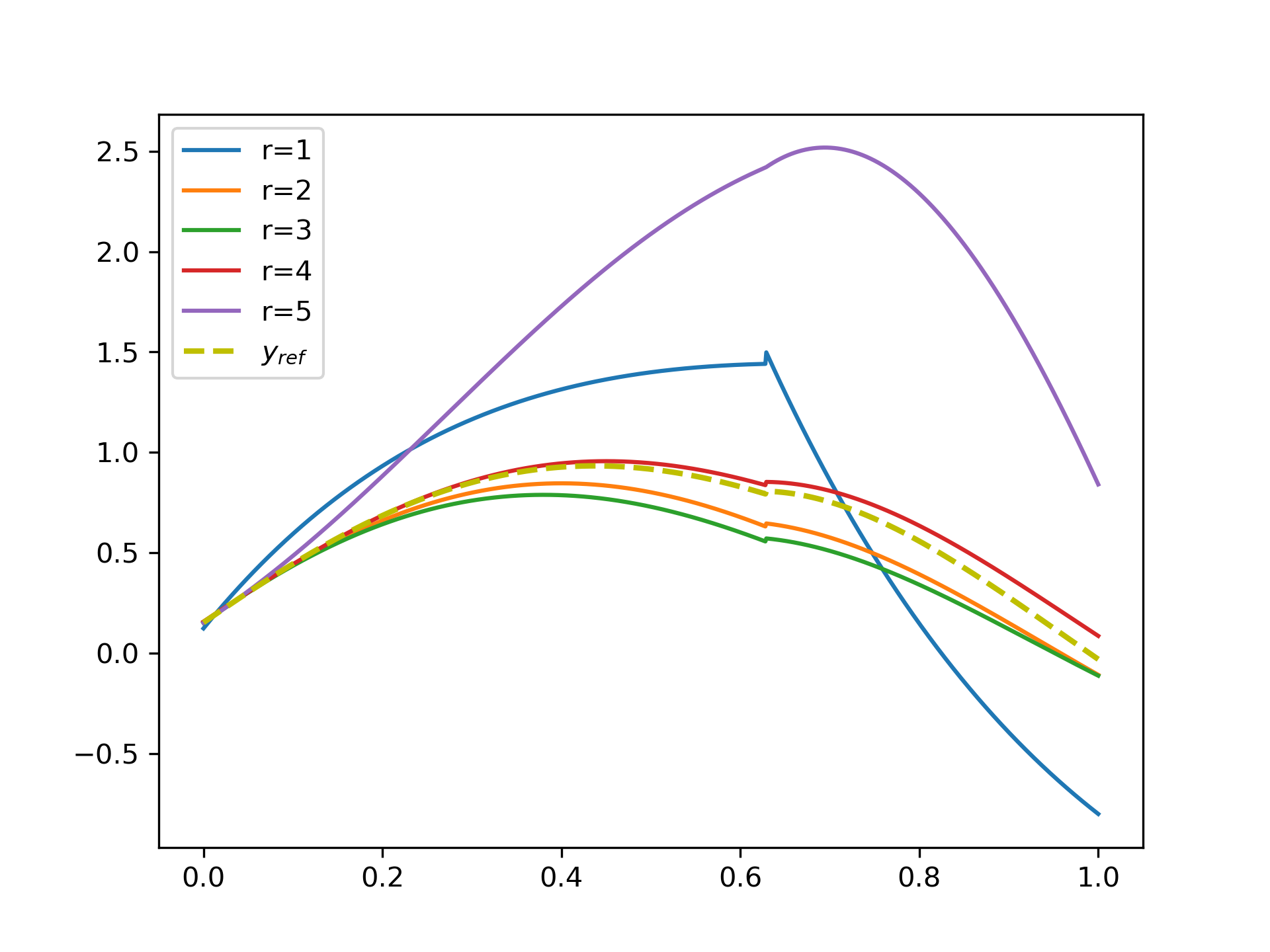}
    \caption{Left: First component of the output obtained in the cross-validation for various $r$. Right: Second component of the output obtained in the cross-validation for various $r$ for $r_{\rm ref}=4$.}
    \label{fig:ex1_cross}
\end{figure}
\end{example}


\begin{example}[System with two Differential Variables]
We used the same setup as in the first test.
This time, 
we generated the $\yd$ output using a pH-DAE system with $r_\text{data}=2$ differential variables. 
As before, we solved the optimization problem for different values of $r$ and compared the corresponding results. 
Figure~\ref{fig:ex2_cost} shows the evolution of the cost and Table~\ref{tab:cost_ex2} shows the costs and errors for the first and last iterations.
As before, we observe that the cost for the system with $r=1$ decreases only marginally.
In contrast, the cost reduction is again greater than 99\,\% for all other systems. 
Figure~\ref{fig:ex2-output} shows the output obtained with the initial matrices before calibration in blue. The orange graph shows the output of the fitted system and the green dashed line shows the output obtained with the reference data.
As the error values above indicated, the fitted output closely matches the reference output.

\begin{figure}[ht!]
    \centering
    \includegraphics[width=0.8\textwidth]{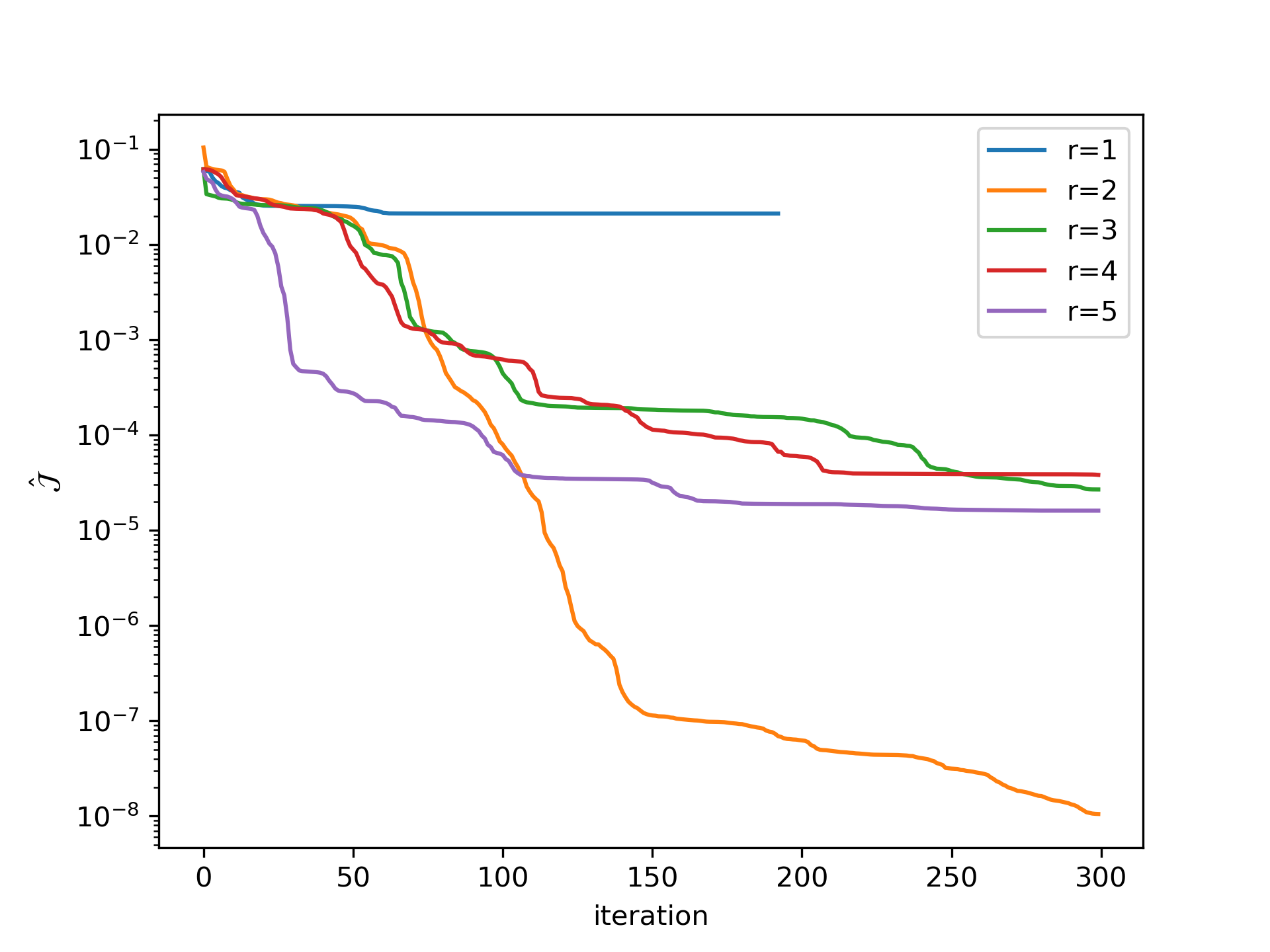}
    \caption{Evolution of the cost for various $r$ over the iterations of the calibration in log plot for $r_{\rm ref}=2$}
    \label{fig:ex2_cost}
\end{figure}

Table~\ref{tab:ex2_cross} and Figure~\ref{fig:ex2_cross} show the results for the cross-validation test. 
The fitted system with $r=2$ provides the best results, yielding the lowest cost and the smallest errors.

In summary, we successfully identified pH-DAE systems using our approach.
However, the results depend heavily on the choice of the hyperparameter~$r$, at least in the cross-validation test.
We note that our problem does not have a unique solution and can therefore depend heavily on the initial data $v^0$.
In other words,  two different surrogate systems that minimize the cost functional may respond differently to the test signal. 

\begin{table}[h]
    \centering
    \begin{tabular}{c|c|c|c|c|c}
        number of differential variables $r$ & 1 & 2 & 3 & 4 & 5\\ \hline
        cost in first iteration &\num{5.96e-02} &\num{1.04e-01} &\num{6.18e-02} &\num{6.17e-02} &\num{5.78e-02} \\
        cost in last iteration &\num{2.12e-02} &\num{1.05e-08} &\num{2.68e-05} &\num{3.82e-05} &\num{1.61e-05}  \\
        $\|y - y_{\rm ref}\|_\infty$ in first iteration &\num{7.84e-01} &\num{9.70e-01} &\num{8.74e-01} &\num{7.94e-01} &\num{7.96e-01}\\
        $\|y - y_{\rm ref}\|_\infty$ in last iteration &\num{4.81e-01} &\num{2.41e-04} &\num{1.20e-02} &\num{1.59e-02} &\num{9.43e-03}  \\
    \end{tabular}
    \caption{Initial and final cost values and the maximum norm between the fitted output and reference output for various numbers of differential variables $r$ for $r_{\rm ref}=2$.}
    \label{tab:cost_ex2}
\end{table}

\begin{table}[h]
    \centering
    \begin{tabular}{c|c|c|c|c|c}
        number of differential variables $r$ & 1 & 2 & 3 & 4 & 5\\ \hline
        $\|y - y_\text{test}\|_L^2$ &\num{3.96e+00} &\num{3.68e-02} &\num{7.25e-01} &\num{2.84e-01} &\num{1.25e+00}  \\
        $\|y - y_\text{test}\|_\infty$  &\num{5.75e+00} &\num{6.71e-02} &\num{1.06e+00} &\num{3.80e-01} &\num{1.85e+00}   \\
    \end{tabular}
    \caption{Cross validation cost values and the maximum norm between the fitted output and reference output for various numbers of differential variables $r$ for $r_{\rm ref}=2$.}
    \label{tab:ex2_cross}
\end{table}

\begin{figure}[ht!]
    \centering
    \includegraphics[width=0.45\textwidth]{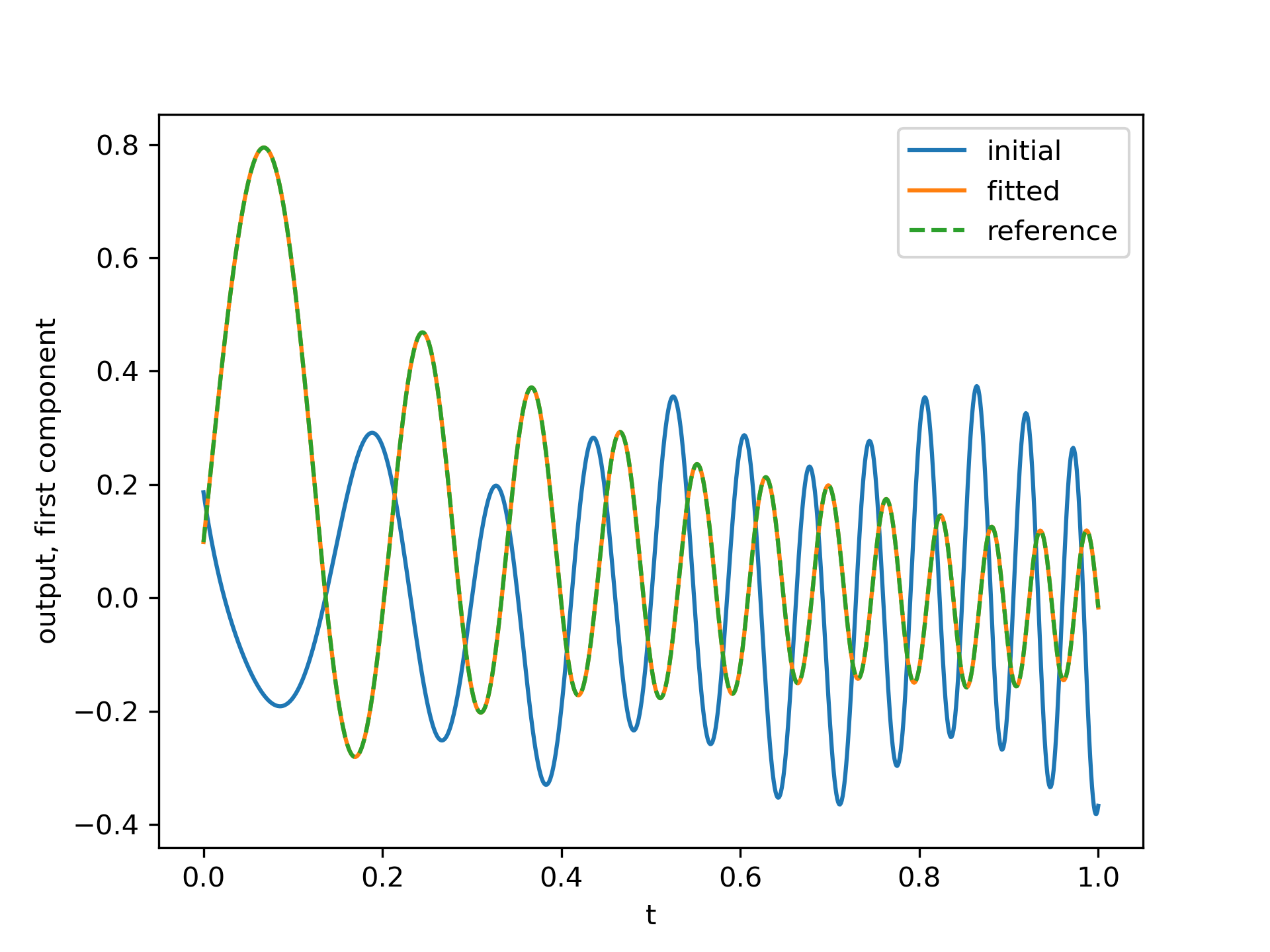}
    \includegraphics[width=0.45\textwidth]{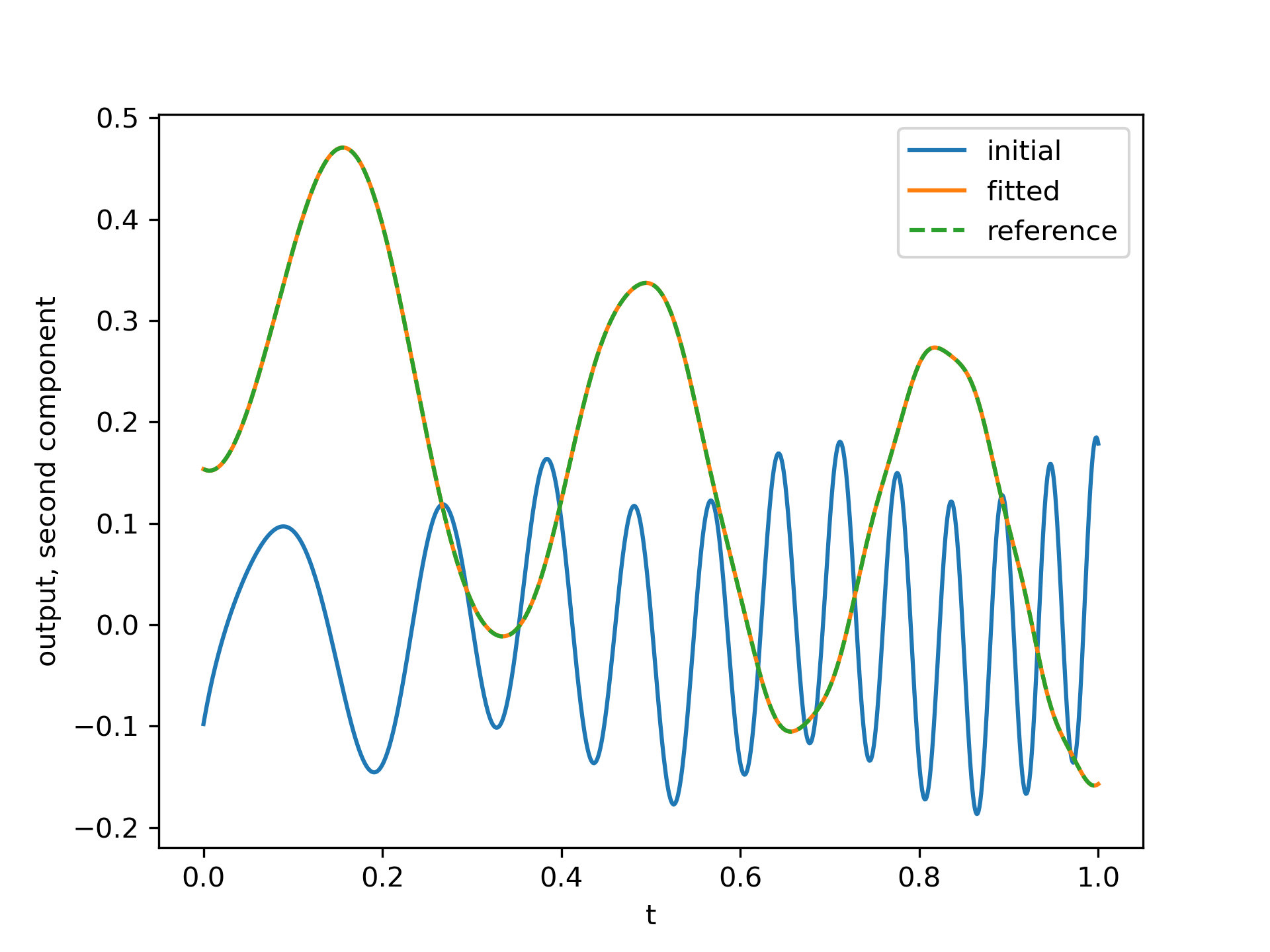}
    \caption{Left: First component of the output for the before the calibration (initial), after the calibration (fitted) and for the reference matrices (reference) used for the generation of the synthetic data. 
    Right: Second component of the output for the before the calibration (initial), after the calibration (fitted) and for the reference matrices (reference) used for the generation of the synthetic data for $r_{\rm ref}=2$.}
    \label{fig:ex2-output}
\end{figure}

\begin{figure}[ht!]
    \centering
    \includegraphics[width=0.45\textwidth]{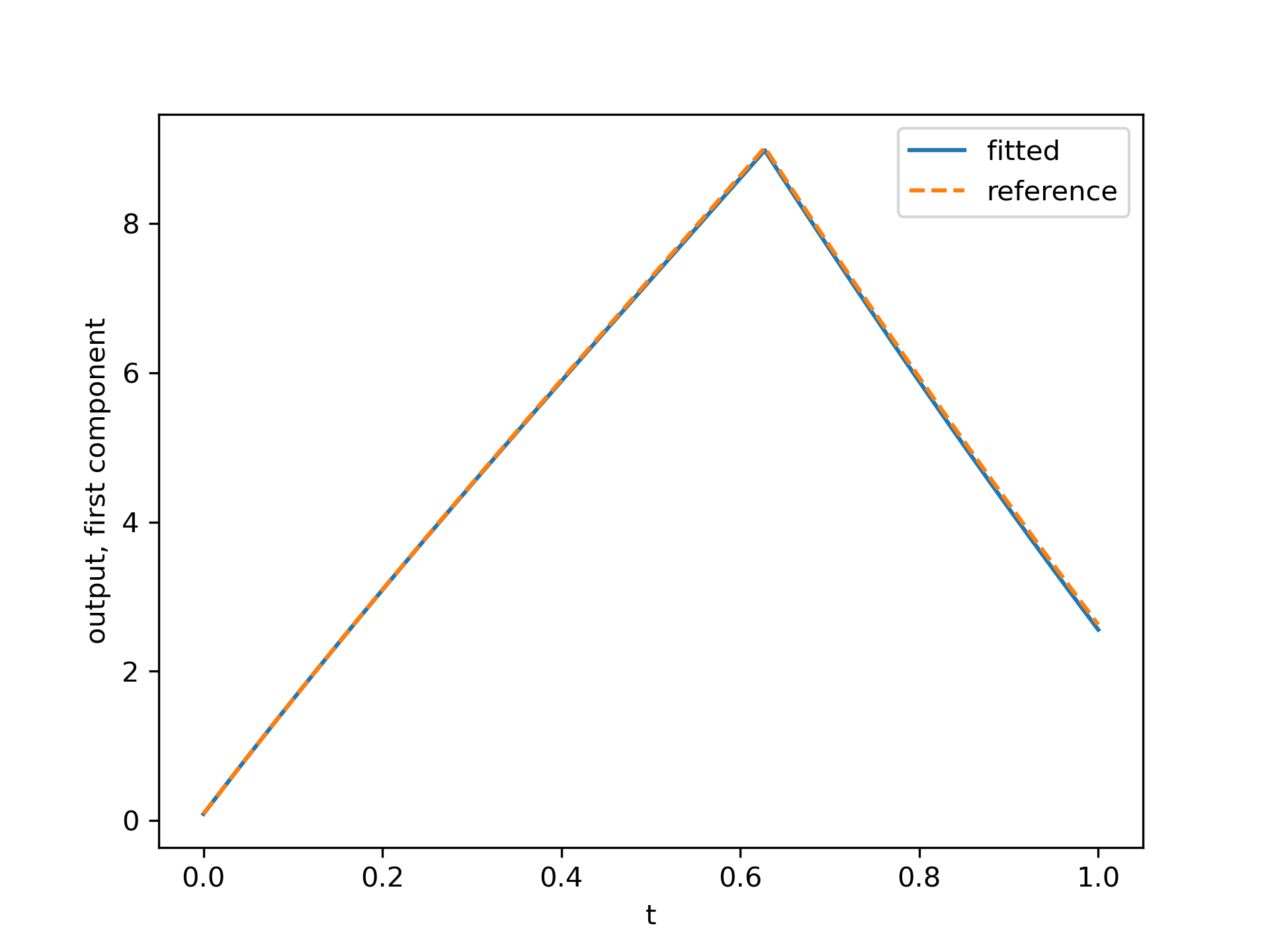}
    \includegraphics[width=0.45\textwidth]{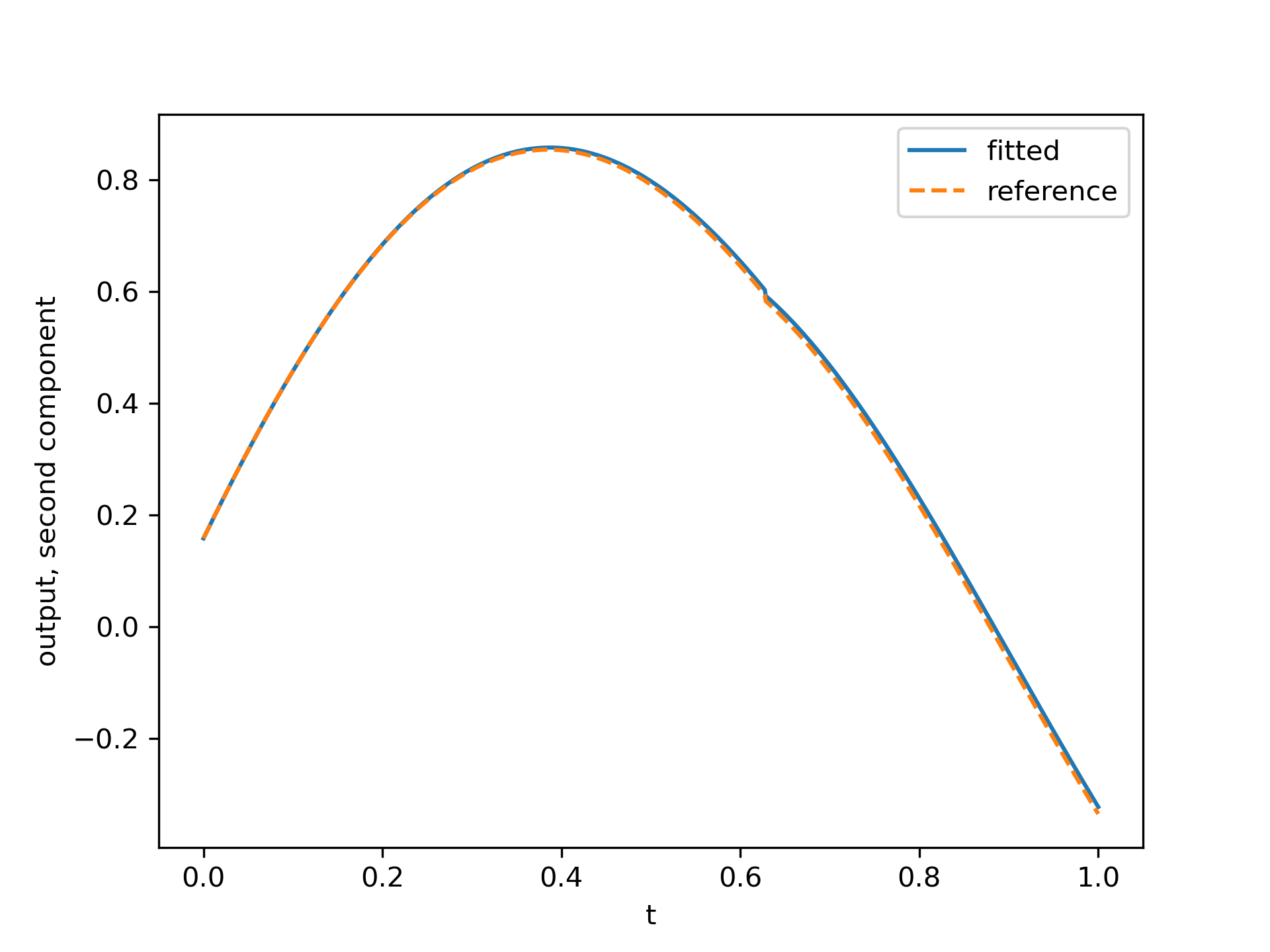}
    \caption{Left: First component of the output obtained in the cross-validation. Right: Second component of the output obtained in the cross-validation for $r_{\rm ref}=2$.}
    \label{fig:ex2_cross}
\end{figure}
\end{example}

\subsection{Model order reduction test with a RCL Ladder Network} 
The last test case we consider is taken from the PHS benchmark collection \cite{schwerdtnerAlgopaulPortHamiltonianBenchmarkSystemsjl2026} and generates pH-DAE models for electrical circuits consisting of ideal voltage sources, resistors, inductors and capacitors. Modeling of such circuits with directed graphs 
(see Figure~\ref{fig:rcl-network}) leads to systems of the following form:
\begin{align*}
    E \dot{x}(t) &= (J-R) x(t) + Bu(t),\\
    y(t) &= B^\top x(t),
\end{align*}
where 
\begin{align*}
    E=\begin{pmatrix}
        \mathcal{A}_c C \mathcal{A}_c^T & 0 & 0\\
        0 & L & 0 \\
        0 & 0 & 0
    \end{pmatrix},
    J=\begin{pmatrix}
        0 & -\mathcal{A}_l & -\mathcal{A}_v\\
        \mathcal{A}_l^\top & L & 0 \\
        \mathcal{A}_v^\top & 0 & 0
    \end{pmatrix},
    E=\begin{pmatrix}
        \mathcal{A}_r R^{-1} \mathcal{A}_r^T & 0 & 0\\
        0 & 0 & 0 \\
        0 & 0 & 0
    \end{pmatrix},
    B=\begin{pmatrix}
        0 \\
        0 \\
        -I_2
    \end{pmatrix}.
\end{align*}
Here, the matrices $R$, $L$, and $C$ are positive definite diagonal matrices containing the resistances, inductances and capacitances as entries.
The incidence matrices $\mathcal{A}_r$, $\mathcal{A}_l$, $\mathcal{A}_c$, and $\mathcal{A}_v$ follow directly from the directed graph of the network and contain only entries in $\{-1,0,1\}$. 
The system's inputs $u(t)$ are the voltages $v_v(t)$ provided by the voltage sources and the outputs are the negative currents $i_v(t)$ through the voltage sources.
The state vector is 
$x(t)\coloneq \begin{psmallmatrix}v(t)^\top & i_l(t)^\top & i_v(t)^\top \end{psmallmatrix}^\top$ 
with node voltages $v(t)$ and inductor currents $i_l(t)$.
Thus, the model dimension is 
$n=3\bar{n}+4$, where $\bar{n}$ denotes the number of loops. 
Here, we choose the value $\bar{n}=32$
and set all resistances to $0.2$, and capacitances and inductances to $1$. 

\begin{figure}[ht!]
    \centering
    \begin{circuitikz}[european resistors, european currents]
        \draw (0,0) to[sV, l=$u_1(t)$] (0,3) 
                    to[short, i=$y_1(t)$] (2,3);
        \draw (2,3) to[R=$R_0$] (2,0);
        \draw (2,3) to[R=$R_1$] (4.5,3)
                    to[L=$L_1$] (6.5,3)
                    to[C, l_=$C_1$] (6.5,0)
                    -- (0,0);   
        \draw(6.5,3) -- (6.8,3);        
        \draw(6.5,0) -- (6.8,0); 
        \draw[loosely dotted, thick] (6.8,3) -- (7.5,3);
        \draw[loosely dotted, thick] (6.8,0) -- (7.5,0);
        \draw (7.5,3) -- (8,3)
            to[C, l=$C_{\bar{n}-1}$] (8,0)
            -- (7.5,0);
        \draw (8,3) to[R=$R_{\bar{n}}$] (10.5,3)
                    to[L=$L_{\bar{n}}$] (12.5,3)
                    to[R, l_=$R_{\bar{n}+1}$] (12.5,0)
                    to[short, i=$y_2(t)$] (10.5,0)
                    to[sV,sources/symbol/rotate=auto, l=$u_2(t)$] (8,0);

    \end{circuitikz}
    \caption{A schematic of the considered RCL Ladder Network, cf.\  \cite{schwerdtnerAlgopaulPortHamiltonianBenchmarkSystemsjl2026}.}
    \label{fig:rcl-network}
\end{figure}
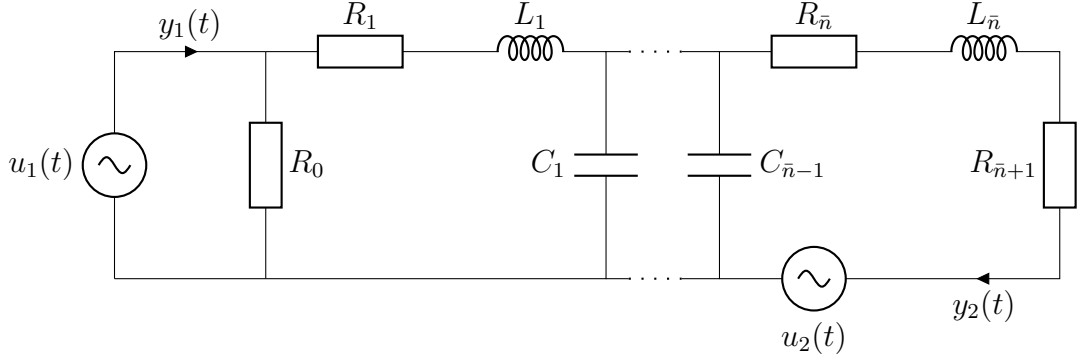

Due to the simple choice of $L$ and $C$, the system is already in semi-explicit form \eqref{eq:pH-DAE_transformed}.  
Hence, the reference data $y_\text{data}$ can be computed directly using the implicit midpoint rule, and the number of differential variables is given by $r=\operatorname{rank} E$.
We took a $2n^2+mn+r \times 1$ vector with entries in $[-1,1]$ as the initial value $v^0$ and a $r\times 1$ vector with entries in $[0,1]$ as the initial value $x_{1,0}$. 
Our objective was to investigate whether a reduced system exhibiting the same behavior as the original system could be found. 
In doing so, we 
examined 
various ratios $r/n$. 
Figure~\ref{fig:ex3_cost} shows the evolution of the costs for $12$ test cases, i.e., $N\in\{100,50,20\}$ and $r/n \in \{1,0.9,0.6,0.3\}$.
Table~\ref{tab:ex3_rel_err} shows that all costs are smaller than \num{1e-5} and that the relative errors after the calibration are smaller than 1\,\%,  except for the 
3 cases where $N=r$. 
Figure \ref{fig:ex3_output} confirms that the outputs only match well for the real DAE cases, i.e., $r<n$.
To conclude, we will now consider the cross-validation test. 
The results are shown in Table~\ref{tab:ex3-cross} and Figure~\ref{fig:ex3_cross}. 
As can be seen here, our pH-DAE system cannot be represented by an ODE system. However, the outputs of the other systems are close to those of the original system.


\begin{figure}[H]
    \centering
    \includegraphics[width=0.62\textwidth]{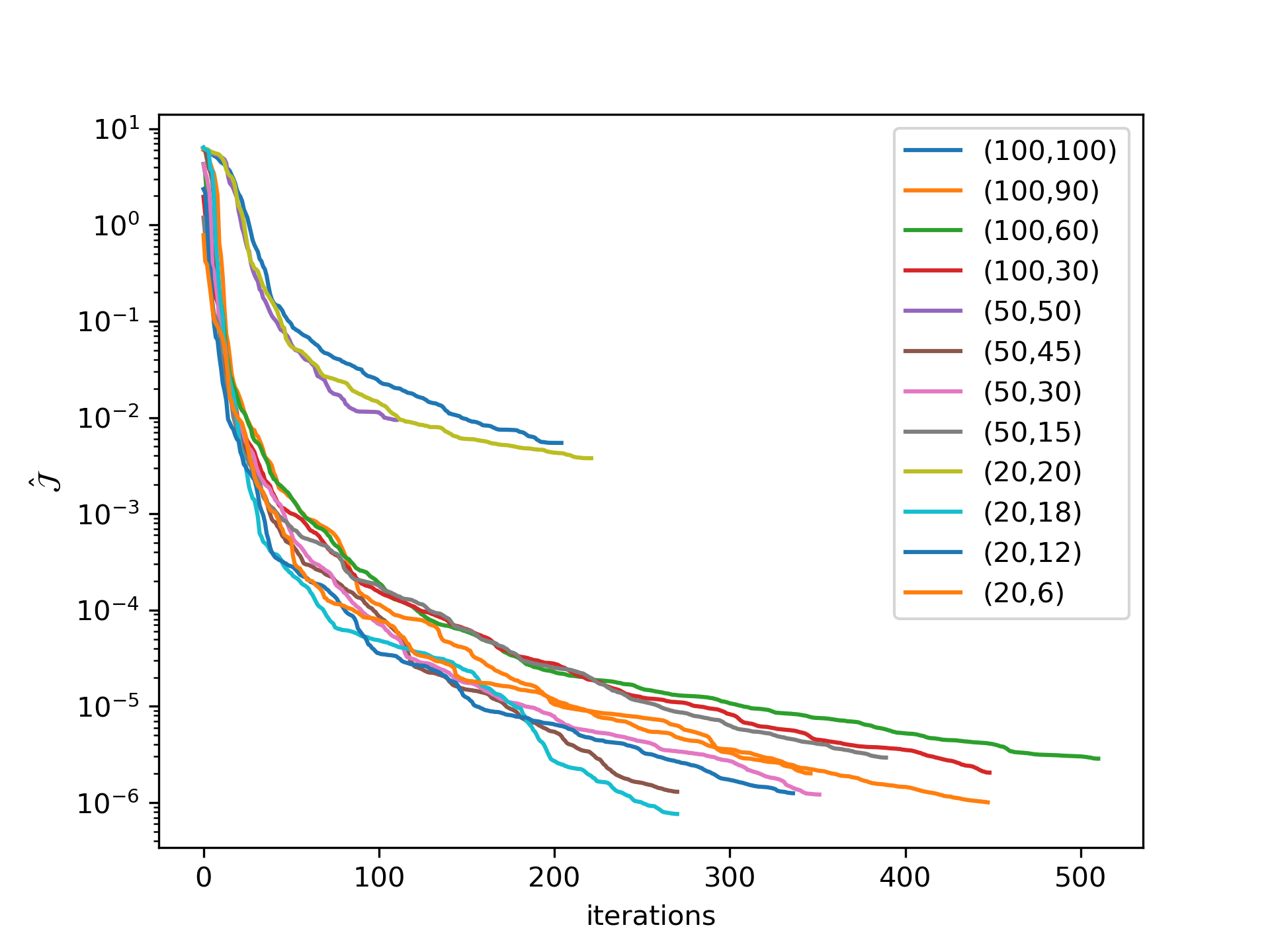}
    \caption{Evolution of the cost for various $r$ and $n$ 
    for $N_{\rm ref} =100$ and $r_{\rm ref}=63$.}
    \label{fig:ex3_cost}
\end{figure}

\begin{table}[H]
    \centering
    \begin{tabular}{c|c|c|c|c} \small
        $ n $ & $r$ & cost in last iter. &$ \|y-y_{\rm ref}\|_\infty/ \|y_{\rm ref}\|_\infty$ before & $ \|y-y_{\rm ref}\|_\infty/ \|y_{\rm ref}\|_\infty$ after\\ \hline
            100 & 100 & \num{5.45e-03}& \num{115}\% & \num{12.18}\% \\
            100 & 90 & \num{1.01e-06}& \num{106}\% & \num{0.10}\% \\
            100 & 60 & \num{2.86e-06}& \num{135}\% & \num{0.16}\% \\
            100 & 30 & \num{2.05e-06}& \num{124}\% & \num{0.08}\% \\
            50 & 50 & \num{9.44e-03}& \num{100}\% & \num{21.87}\% \\
            50 & 45 & \num{1.30e-06}& \num{159}\% & \num{0.11}\% \\
            50 & 30 & \num{1.21e-06}& \num{118}\% & \num{0.08}\% \\
            50 & 15 & \num{2.93e-06}& \num{90}\% & \num{0.08}\% \\
            20 & 20 & \num{3.78e-03}& \num{128}\% & \num{13.95}\% \\
            20 & 18 & \num{7.61e-07}& \num{135}\% & \num{0.08}\% \\
            20 & 12 & \num{1.25e-06}& \num{105}\% & \num{0.11}\% \\
            20 & 6 & \num{2.01e-06}& \num{120}\% & \num{0.09}\% \\
    \end{tabular}
    \caption{Relative errors of the output measured before and after the calibration.}
    \label{tab:ex3_rel_err}
\end{table}

\begin{figure}[H]
    \centering
    \includegraphics[width=0.42\textwidth]{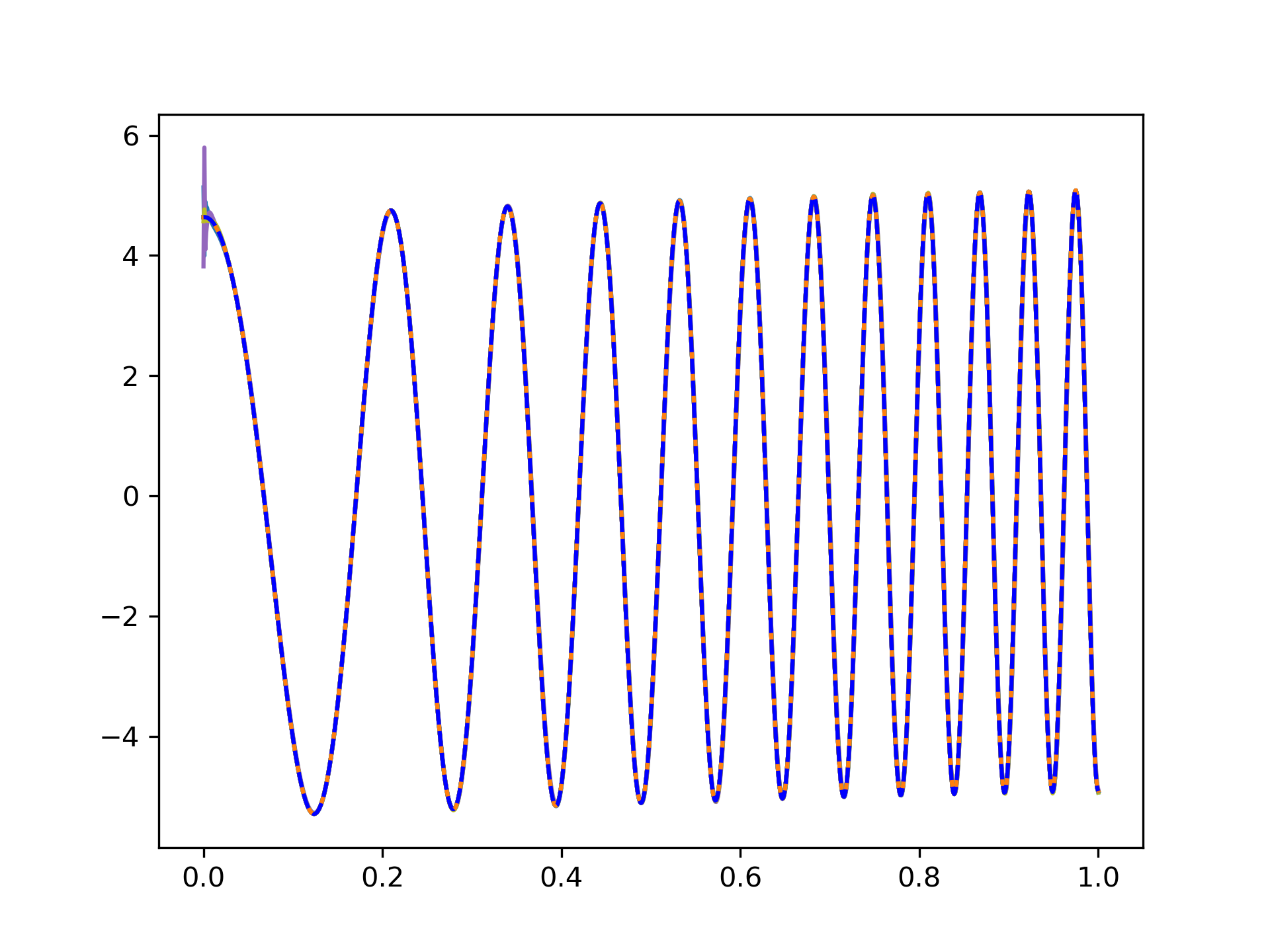}
    \includegraphics[width=0.42\textwidth]{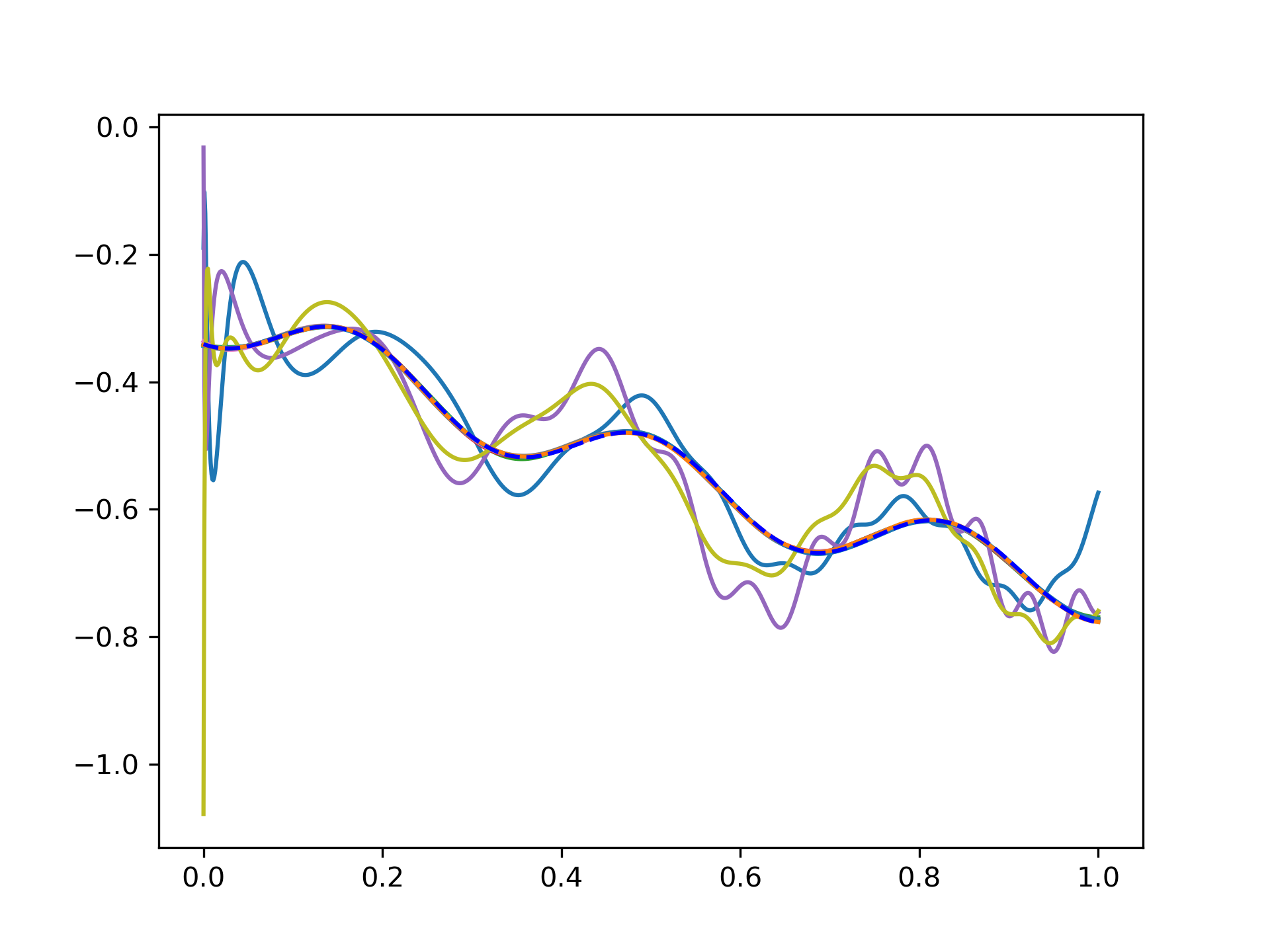}
    \caption{\small Left: first component of the fitted output for different 
    $n$, $r$, and $y_{\rm ref}$
    (blue dashed line). 
    Right: second component of the fitted output for different $n$, $r$, and $y_{\rm ref}$ (blue dashed line).}
    \label{fig:ex3_output}
\end{figure}

\begin{figure}[H]
    \centering
    \includegraphics[width=0.42\textwidth]{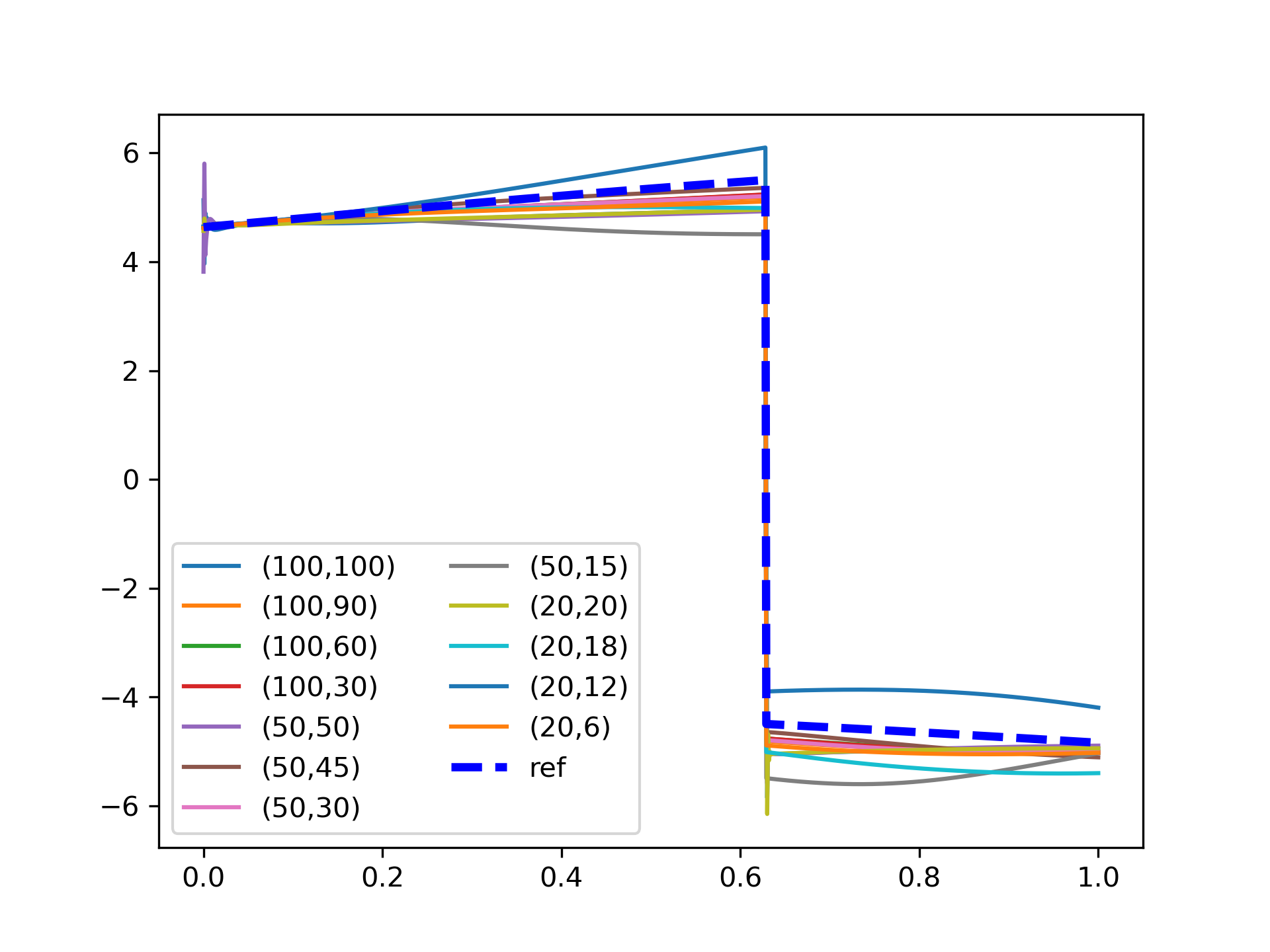}
    \includegraphics[width=0.42\textwidth]{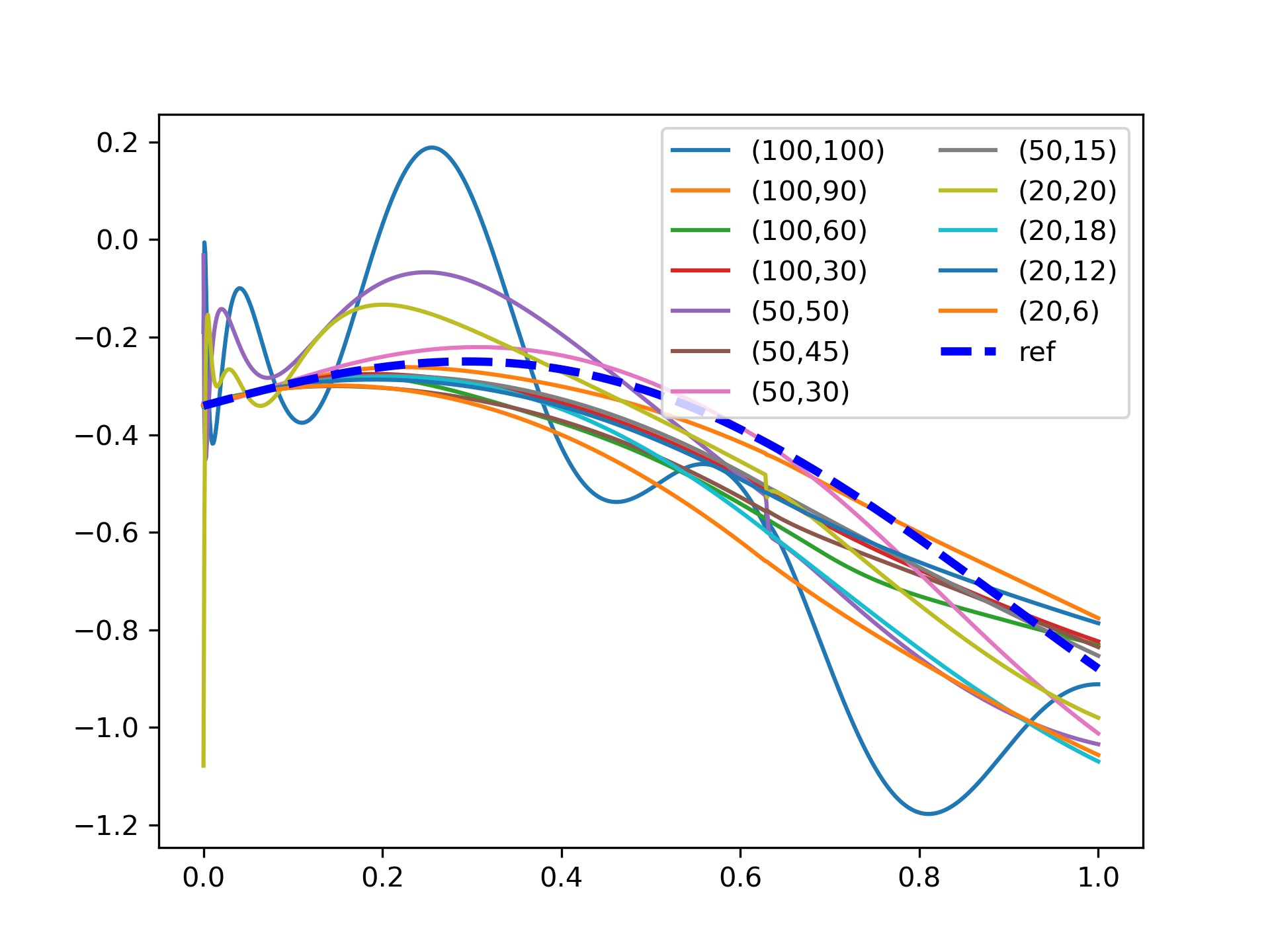}
    \caption{\small Left: first component of the fitted output for different $n$, $r$, and $y_{\rm ref}$ (blue dashed line) 
    Right: second component of the fitted output for different $n$, $r$, and $y_{\rm ref}$ (blue dashed line).} 
    \label{fig:ex3_cross}
\end{figure}

\begin{table}[H]
    \centering
    \begin{tabular}{c|c|c|c}
        $ n $ & $r$ &$ \|y-y_{\rm test}\|_{L^2((0,1),\R^2)}$ & $ \|y-y_{\rm test}\|_\infty $  \\ \hline
        100 & 100 & \num{0.46} & \num{3.36} \\
        100 & 90 & \num{0.23} & \num{0.33} \\
        100 & 60 & \num{0.22} & \num{0.29} \\
        100 & 30 & \num{0.21} & \num{0.30} \\
        50 & 50 & \num{0.40} & \num{3.74} \\
        50 & 45 & \num{0.18} & \num{0.28} \\
        50 & 30 & \num{0.22} & \num{0.33} \\
        50 & 15 & \num{0.66} & \num{1.03} \\
        20 & 20 & \num{0.36} & \num{2.81} \\
        20 & 18 & \num{0.45} & \num{0.67} \\
        20 & 12 & \num{0.50} & \num{0.77} \\
        20 & 6 & \num{0.33} & \num{0.41} \\
    \end{tabular}
    \caption{Difference of fitted output and reference output in maximum /  $L^2$ norm in cross-validation test.}
    \label{tab:ex3-cross}
\end{table}
\section{Conclusion}\label{sec:conclusion}

We presented a data-driven approach for identifying linear index-1 port-Hamiltonian differential-algebraic equation (pH-DAE) systems from input-output data. 
By incorporating the pH structure into the parameterization, the identified models retain the relevant structural properties. 
The index-1 structure is exploited to formulate the identification problem as an unconstrained optimization problem, while an adjoint formulation provides the gradient of the reduced cost functional efficiently.

Numerical experiments demonstrate that the proposed approach can construct pH-DAE surrogate models that accurately reproduce the input-output behaviour of the reference systems. 
The cross-validation experiments further indicate that the identified models can generalize to independent input signals.
The experiments also show the potential for constructing lower-dimensional structure-preserving surrogate models.

Future work will focus on extending the approach to higher-index and nonlinear pH-DAE systems. 
Further directions include the treatment of noisy and incomplete measurements, uncertainty quantification, and the systematic combination of identification with structure-preserving model reduction for large-scale multiphysics applications.


\section*{Acknowledgements} 
\noindent
M.~Ehrhardt, M. Günther and C. Totzeck acknowledge funding by the Deutsche Forschungsgemeinschaft (DFG, German Research Foundation) -- Project-ID 531152215 -- CRC 1701.

\section*{Author contributions} 
\begin{description} 
    \item[TZ] Investigation, Software, Validation, Visualization, Writing – original draft
    \item[ME] Investigation, Project administration, Supervision, Writing – review and editing
    \item[MG]  Conceptualization, Methodology, Writing – review and editing
   \item[CT] Methodology, Supervision, Writing – review and editing
\end{description}




\end{document}